\documentclass[10pt]{article}
\usepackage{amsfonts, epsfig, amsmath, amssymb,amsthm, color,mathabx}
\usepackage{mathrsfs}

\usepackage[english]{babel}

\usepackage[round]{natbib}   

\newcommand{\E}{\mathbf{E}}
\renewcommand{\P}{\mathbf{P}}

\renewcommand {\epsilon}{\varepsilon}

\newtheorem{thm}{Theorem}[section]
\newtheorem{prp}[thm]{Proposition}

\newtheorem{lem}[thm]{Lemma}
\theoremstyle{definition}
\newtheorem{dfn}[thm]{Definition}

\newtheorem{rem}[thm]{Remark}

\DeclareMathSymbol{\ophi}{\mathalpha}{letters}{"1E}

\let\oldmarginpar\marginpar
\renewcommand{\marginpar}[1]{\oldmarginpar{\scriptsize\texttt{\color{blue}{#1}}}}

\newcommand{\e}{\varepsilon}

\renewcommand{\phi}{\varphi}

\newcommand{\be}{\begin{equation}}
\newcommand{\ee}{\end{equation}}
\newcommand{\ben}{\begin{equation*}}
\newcommand{\een}{\end{equation*}}

\newcommand{\ba}{\begin{equation}\begin{aligned}}
\newcommand{\ea}{\end{aligned}\end{equation}}
\newcommand{\ban}{\begin{equation*}\begin{aligned}}
\newcommand{\ean}{\end{aligned}\end{equation*}}

\DeclareMathOperator{\sign}{sign}

\DeclareMathOperator{\Law}{Law}

\newcommand{\ex}{\mathrm{e}}
\newcommand{\di}{\mathrm{d}}
\newcommand{\abs}[1]{\left|#1\right|}
\newcommand{\scprod}[1]{\left\langle#1\right\rangle}

\newcommand{\rF}{\mathscr{F}}

\newcommand{\bF}{\mathbb{F}}

\newcommand{\bI}{\mathbb{I}}

\newcommand{\bR}{\mathbb{R}}

\newfont{\cyrfnt}{wncyr10}
\def\J3{\cyrfnt{\rm \u{\cyrfnt I}}}
\def\j3{\cyrfnt{\rm \u{\cyrfnt i}}}

\allowdisplaybreaks[4]

\numberwithin{equation}{section}

\newcommand{\ansref}[2]{#2}
\newcommand{\chng}[2][]{ #2}

\begin{document}
\title{A Stratonovich SDE with irregular coefficients: \\Girsanov's example revisited}


\date{\today}

\author{Ilya Pavlyukevich\footnote{Institute for Mathematics, Friedrich Schiller University Jena, Ernst--Abbe--Platz 2, 07743 Jena, Germany;
\texttt{ilya.pavlyukevich@uni-jena.de}}\ \  and Georgiy Shevchenko\footnote{Department of Probability Theory, 
 Statistics and Actuarial Mathematics, Taras Shevchenko National University of Kyiv, Volodymyrska 64, Kiev 01601, Ukraine; 
\texttt{zhora@univ.kiev.ua}}
}

\maketitle


\begin{abstract}
In this paper we study the Stra\-to\-no\-vich stochastic differential equation $\di X=|X|^{\alpha}\circ\di B$, $\alpha\in(-1,1)$,
which has been introduced by Cherstvy et al.\ [New Journal of Physics 15:083039 (2013)] in the context of analysis of anomalous diffusions 
in heterogeneous media.
We determine its weak and strong solutions, which are homogeneous strong Markov processes 
\chng{spending zero time at $0$: for $\alpha\in (0,1)$, these solutions have the form 
\ban
X_t^\theta=\bigl((1-\alpha)B_t^\theta\bigr)^{1/(1-\alpha)},
\ean
where $B^\theta$ is the $\theta$-skew 
Brownian motion driven by $B$ and starting at $\frac{1}{1-\alpha}(X_0)^{1-\alpha}$, $\theta\in [-1,1]$,}
and $(x)^{\gamma}=|x|^\gamma\sign x$; for $\alpha\in(-1,0]$, only the case $\theta=0$ is possible.
The central part of the paper consists in the proof of the existence of a quadratic covariation $[f(B^\theta),B]$ for a locally square integrable
function $f$ and is based on the time-reversion technique for Markovian diffusions.\smallskip


\smallskip

\end{abstract}

\section{Introduction}

\citet{Girsanov-62} considered the It\^o stochastic differential equation
\ba
\label{e:2}
Y_t=Y_0+\int_0^t |Y_s|^\alpha\, \di B_s,\quad t\ge 0,
\ea
driven by a standard Brownian motion $B$ as an example of an SDE with non-unique solution.
In particular, it was shown that for $\alpha\in(0,1/2)$, equation \eqref{e:2} has 
infinitely many continuous strong Markov (weak) solutions as well as non-homogeneous Markov solutions; non-Markovian solutions can also be constructed. 

Since then, equation \eqref{e:2} serves as a benchmark example of various peculiar effects which come to 
light when one weakens the standard regularity assumptions on the coefficients of an SDE.

The proof of weak uniqueness for $\alpha\geq 1/2$ and non-uniqueness for $\alpha\in (0,1/2)$ with the help of random time change
was given by \cite[\S 3.10b]{McKean-1969} whereas a construction of an uncountable set of weak solutions can be found in 
\cite[Example 3.3]{EngelbertSchmidt-85}. The existence and uniqueness of a strong solution
for $\alpha\in[1/2,1]$ was established by 
\cite[Theorem 4]{zvonkin74}. 

Furthermore for $\alpha\in(0,1/2)$, it was shown in 
\cite[Theorem 5.2]{EngelbertSchmidt-85} that for every initial value $Y_0\in\bR$, there is a weak
solution to \eqref{e:2} that spends zero time at $0$ (the so-called fundamental solution) and the law of such a solution is
unique. Path-wise uniqueness among those
solutions to \eqref{e:2} that spend zero time at $0$, and existence of a strong solution was proven by 
\cite[Theorem 1.2]{BassBC-07}.

An analogue of \eqref{e:2} with the \emph{Stratonovich integral}
\ba
\label{e:1}
X_t=X_0+\int_0^t |X_s|^\alpha\circ \di B_s
\ea
was recently introduced in the physical literature
by \cite{CherstvyCM-13}  under the name \emph{heterogeneous diffusion process}. The authors studied the autocorrelation 
function of this process analytically and investigated its sub- and super-diffusive behaviour with the help of numerical simulations.
A similar system with an additional linear drift was considered earlier by \cite{DenHor-02}.

In this paper we will further investigate equation \eqref{e:1} with $\alpha\in (-1,1)$.
It turns out that equation \eqref{e:1} has properties quite different from its It\^o counterpart.   
Let us first make some observations about it. The only problematic point of the diffusion coefficient 
$\sigma(x) = |x|^\alpha$ is $x=0$. For  $\alpha\in (0,1)$, the Lipschitz continuity fails at this point, and for $\alpha\in(-1,0)$ even the 
continuity and boundedness. However one can easily solve \eqref{e:1} locally for initial points $X_0\neq 0$. 

Indeed, assume for definiteness that $X_0>0$. For any $\e>0$, using the properties of the Stratonovich integral, we see that the process given by
\ba
\label{e:X}
X_t^0=\Big((1-\alpha)B_t+X_0^{1-\alpha}\Big)^\frac{1}{1-\alpha}.
\ea
solves equation \eqref{e:1} until the time 
$
\tau_{\e}=\inf\{t\geq 0\colon X_t = \e\} 
$.
Moreover, the solution is unique until $\tau_{\e}$. Consequently, the formula \eqref{e:X} defines a unique strong solution until the time 
$
\tau_{0}=\inf\{t\geq 0\colon X_t^0 = 0\}$
when the process first hits zero. It is clear that extending the solution by zero value beyond 
$\tau_0$ gives a strong solution. However, the uniqueness fails: as it is shown in Section~\ref{sec:bm}, 
the formula \eqref{e:X} defines a strong solution (called a benchmark solution), if we understand 
\ansref{2}{the right-hand side} as a signed power function, 
i.e.\  $(x)^{\frac{1}{1-\alpha}} = \abs{x}^{\frac{1}{1-\alpha}} \sign x$. 
In Section~\ref{sec:bm}, we also construct some non-Markov solutions of \eqref{e:1}. 

\chng[R2: elaborated on the structure]{The next question is whether, as for the It\^o equation, uniqueness holds within the class of 
solutions spending zero time at $0$; this question is addressed in Section~\ref{sec:0@0}. For $\alpha\in (0,1)$, this question is answered negatively: 
in Theorem \ref{t:main} 
we show that equation \eqref{e:1} has also `skew' solutions 
\ban
X_t^\theta=\Big|(1-\alpha)B^\theta_t\Big|^\frac{1}{1-\alpha}
\cdot \sign\Big((1-\alpha)B^\theta_t \Big),
\ean
where for $\theta\in[-1,1]$, $B^\theta$ is the skew Brownian motion, which solves the stochastic 
differential equation $B^\theta_t =\frac{1}{1-\alpha}(X_0)^{1-\alpha}+ B_t+\theta L_t(B^\theta)$, $L$ being the symmetric local time at $0$.
Moreover, we show all solutions, which are homogeneous strong Markov processes and which spend zero time at $0$, have this form.

In Section~\ref{sec:etc}, we propose an explanation for a diversity of strong solutions to \eqref{e:1} and discuss further questions regarding the equation. 
Sections~\ref{sec:six} and \ref{sec:seven} contain the proofs of results concerning weak solutions to  equation \eqref{e:1}, Section~\ref{s:Y} contains the proof of existence of the bracket $[f(B^\theta),B]$, and Section~\ref{sec:nine} is devoted to the proof of the main result concerning strong solutions.
}

\section{Preliminaries and conventions} \label{sec:prelim}

Throughout the article, we work on a stochastic basis $(\Omega,\rF,\bF,\P)$, i.e.\ a complete 
probability space with a filtration $\bF = (\rF_t)_{t\ge 0}$ satisfying the standard assumptions. 
The process $B = (B_t)_{t\ge 0}$ is a standard continuous Brownian motion on this stochastic basis. 

First we briefly recall definitions related to stochastic integration. More details may be found in \cite{Protter-04}. 

The main mode of convergence considered here is the uniform convergence on compacts in probability  (the u.c.p.\ convergence for short): 
a sequence $X^n = (X_t^n)_{t\ge 0}, n\ge 1$, of stochastic processes converges to $X = (X_t)_{t\ge 0}$ in u.c.p.\ if for any $t\ge 0$ 
$$
\sup_{s\in[0,t]} \abs{X^n_s - X_s} \overset{\P}{\longrightarrow} 0, n\to\infty.
$$

Let a sequence of partitions $D_n=\{0=t^n_0<t_1^n<t_2^n<\cdots\}=\{0=t_0<t_1<t_2<\cdots\}$ of $[0, \infty)$ 
be such that for each $t\geq 0$ the number of points in each interval $[0, t]$ is finite, 
and  $\|D_n\|:=\sup_{k\ge 1}\abs{t_{k}^n - t_{k-1}^n}\to 0$ as $n\to \infty$.
\ansref{2}{A continuous process $X$ has quadratic variation $[X]$ if the limit
\ban
{}[X]_t:=\lim_{n\to \infty} \sum_{t_k\in D_n, t_k<t}(X_{t_{k+1}}-X_{t_{k}})^2
\ean
exists in the u.c.p.\ sense.
Similarly, the quadratic 
covariation $[X,Y]$ of two continuous processes $X$ and $Y$ is defined as a limit in u.c.p.\
\ban
{}[X,Y]_t:=\lim_{n\to \infty} \sum_{t_k\in D_n, t_k<t}(X_{t_{k+1}}-X_{t_{k}})(Y_{t_{k+1}}-Y_{t_{k}}).
\ean
} 
When $X$ and $Y$ are semimartingales, the quadratic variations $[X]$, $[Y]$ and the quadratic 
covariation $[X,Y]$ exist, moreover, they have bounded variation on any finite interval.

Further, we define the It\^o (forward) integral as a limit 
in u.c.p.\
\ban
\int_0^t X_s\, \di Y_s=\lim_{n\to \infty} \sum_{t_k\in D_n, t_k<t} X_{t_{k}}(Y_{t_{k+1}}-Y_{t_{k}})
\ean
and the Stratonovich (symmetric) integral as limit in u.c.p.\
\ban
&\int_0^t X_s\circ \di Y_s= \int_0^t X_s\, \di Y_s+\frac 12 [X,Y]_t  
\\&= \lim_{n\to \infty} \sum_{t_k\in D_n, t_k<t} \frac12 (X_{t_{k+1}}+X_{t_k})(Y_{t_{k+1}}-Y_{t_{k}}),
\ean
provided that both the It\^o integral and the quadratic variation exist. Again, when both $X$ and $Y$ are continuous semimartingales, 
both integrals exists, and the convergence holds in u.c.p.
There is an alternative approach to Stratonovich stochastic integration, developed in 
\cite{russo1993forward,russo1995generalized,russo2000stochastic}, which allows integration
with respect to non-semimartingales and non-Markov processes like fractional Brownian motion.

For a process $X$, by $L_t(X)$ we denote the symmetric local time at zero defined as the limit in probability
\ba
L_t(X)=\lim_{\e \downarrow 0}\frac{1}{2\e}\int_0^t \bI_{[-\e,\e]}(X_s)\,\di s.
\ea

In this paper, we define 
\ban
\sign x=\begin{cases}
            -1,\ x< 0,\\
            0,\ x=0,\\
            1,\ x>0,
           \end{cases}
\ean
and for any $\alpha\in \bR$ we set
\ban
|x|^\alpha=\begin{cases}
            |x|^\alpha,\ x\neq 0,\\
            0,\ x=0.
           \end{cases}
\ean
With this notation, for example, for $\alpha=0$ we have $|x|^0=\bI(x\neq 0)$.
We also denote 
\ban
(x)^\alpha = |x|^\alpha \sign x.
\ean
Throughout the article, $C$ will be used to denote a generic constant, whose value is not important and may change between lines. 

\section{Benchmark solution}\label{sec:bm}

Now we turn to equation \eqref{e:1}. The concept of strong solution is defined in a standard manner.
\begin{dfn}
A strong solution to \eqref{e:1} is a continuous stochastic process $X$ such that 
\begin{enumerate}
\item
$X$ is adapted to the augmented 
natural filtration of $B$;
\item 
for any $t\geq 0$, the It\^o integral $\int_0^t |X_s|^\alpha \di B_s$ and the quadratic covariation 
$[|X|^\alpha,\ B]_t$ exist;
\item 
for any $t\geq 0$, equation \eqref{e:1}  holds $\P$-a.s.
\end{enumerate}
\end{dfn}

Define the \emph{benchmark solution} to equation \eqref{e:1} by
\ba
\label{e:benchmark}
X^0_t=\Big((1-\alpha)B_t+(X_0)^{1-\alpha}\Big)^\frac{1}{1-\alpha}.
\ea
The following change of variable result is crucial for proving that it solves equation \eqref{e:1}. 
\begin{thm}[Theorem 4.1, \cite{FPSh-95}]\label{t:ito}
Let $F$ be absolutely continuous with locally square integrable derivative $f$. Then
\ban
F(B_t) =F(B_0)+ \int_0^t f(B_s)\,\di B_s + \frac12[f(B), B]_t.
\ean
\end{thm}
\begin{thm}
\label{t:exists}
For $\alpha\in (-1,1)$ and $X_0\in \bR$, the process $X^{0}$ given by \eqref{e:benchmark} is a strong solution to \eqref{e:1}.
\end{thm}
\begin{proof}
For each $X_0\in \bR$, we note that the function
\ban
F(x)=\Big((1-\alpha)x-(X_0)^{1-\alpha}\Big)^\frac{1}{1-\alpha}
\ean
is absolutely continuous for $\alpha\in (-1,1)$ and its derivative
\ban
f(x)= \Big|(1-\alpha)x-(X_0)^{1-\alpha}\Big|^\frac{\alpha}{1-\alpha}=|F(x)|^\alpha
\ean
is locally square integrable, so by Theorem~\ref{t:ito} the process $X^{0}=F(B)$ satisfies \eqref{e:1}.
\end{proof}
As explained in the introduction, the benchmark solution is a unique strong solution until the time 
\ba\label{e:tau0}
\tau_{0}=\inf\{t\geq 0\colon X^0_t = 0\}=\inf_{t\geq 0}\Big\{t\geq 0\colon B_t= -\frac{(X_0)^{1-\alpha}}{1-\alpha} \Big\}.
\ea
when it first hits $0$. 
However, the uniqueness fails after this time. Namely, with the help of Theorem \ref{t:ito} one can easily construct other strong solutions; the proofs are the same as for $X^0$ and therefore omitted. One example is the solution stopped at $0$. 
\begin{thm}
For $\alpha\in (-1,1)$ and $X_0\in \bR$, the process 
$$
X'_t=  \Big((1-\alpha)B_t+(X_0)^{1-\alpha}\Big)^\frac{1}{1-\alpha}\bI_{t\le \tau_0},
$$
where  $\tau_0$ is given by  \eqref{e:tau0}, is a strong solution to \eqref{e:1}.
\end{thm}

Both $X^0$ and $X'$ possess the strong Markov property, thanks to that of $B$. One can also construct an uncountable 
family of non-Markov solutions.
Namely, for any  $A,B>0$,
set
\ban
F_{A,B}(x)=\begin{cases}
                 -|x+A|^\frac{1}{1-\alpha},\quad x<-A,\\
                 0,\quad -A\leq x\leq B,\\
                 |x-B|^\frac{1}{1-\alpha},\quad x>B.
                \end{cases}
\ean
and  $$X^{A,B}_t=F_{A,B}\bigl((1-\alpha)B_t+(X_0)^{1-\alpha}\bigr)$$ which equals to zero as long as 
$(1-\alpha)B_t+(X_0)^{1-\alpha}  \in[-A,B]$.  
\begin{thm}
For $\alpha\in (-1,1)$ and $X_0\in \bR$, the process $X^{A,B}$ is a strong solution to \eqref{e:1}.
\end{thm}

\section{Solutions spending zero time at $0$}\label{sec:0@0}

The property of spending zero time at $0$ is known to be crucial to guarantee uniqueness, see 
e.g.\ \cite{Beck-73} for the deterministic differential equations and \cite{BassBC-07,AryasovaP-11} for the stochastic case. We will need also a concept of a weak solution to \eqref{e:1}.

\begin{dfn}
A weak solution of \eqref{e:1} is a pair $(\widetilde X, \widetilde B)$ of adapted continuous processes on a 
stochastic basis $(\widetilde \Omega,\widetilde\rF,\widetilde\bF,\widetilde\P)$ such that
\begin{enumerate}
\item
$\widetilde B$ is a standard Brownian motion on $(\widetilde \Omega,\widetilde\rF,\widetilde\bF,\widetilde\P)$;
\item 
for any $t\geq 0$, the It\^o integral $\int_0^t |\widetilde X_s|^\alpha \di\widetilde B_s$ and the quadratic covariation 
$[|\widetilde X|^\alpha,\widetilde B]_t$ exist;
\item 
for any $t\geq 0$, 
$$
\widetilde X_t=X_0+\int_0^t |\widetilde X_s|^\alpha\circ \di \widetilde B_s
$$
 holds $\P$-a.s.
\end{enumerate}
\end{dfn}
\begin{dfn}
\label{d:00}
A process $X$ is said to spend zero time at $0$ if 
for each $t\geq 0$
\ban
\label{e:00}
\int_0^t \bI_{\{0\}}(X_s)\,\di s=0 \quad \P\text{-a.s.}
\ean
\end{dfn}

In order to define solutions to \eqref{e:1}, different from the benchmark solution \eqref{e:benchmark} and spending 
zero time at $0$, recall the notion of skew Brownian motion. 
\chng{
For $\theta\in[-1,1]$,  the skew Brownian motion $B^\theta=B^\theta(x)$ starting at $x\in\bR$ is the unique solution of the SDE
\ba
\label{e:sdesbm}
B_t^\theta=x+B_t+\theta L_t(B^\theta),
\ea
where $L(B^\theta)$ is the symmetric local time of $B^\theta$ at $0$; 
see \cite{HShepp-81} and \cite[Section 5]{Lejay-06}. 
}
Roughly speaking, the process $B^\theta$ behaves like a standard Brownian motion outside of zero. At zero its decides to evolve in the 
positive or negative directions independently of the past (the strong Markov property) with the ``flipping'' probabilities
$\beta_\pm=\frac{1\pm\theta}{2}$. 
\chng{
For the initial value $x=0$ and $\theta=0$, $B^0\equiv B$, and 
for $\theta=\pm 1$, the solution to the equation \eqref{e:sdesbm} is a reflected Brownian motion starting at zero and can be 
written explicitly, namely
\ban
B^1&=(B^1_t=B_t-\min_{s\leq t}B_s)_{t\geq 0}\stackrel{\di }{=}(|B_t|)_{t\geq 0},\\
B^{-1}&=(B^{-1}_t=B_t-\max_{s\leq t}B_s)_{t\geq 0}\stackrel{\di }{=}(-|B_t|)_{t\geq 0}.
\ean
}
A complete account on the properties of the skew Brownian motion can be found in 
\cite{Lejay-06}.

First we describe the law of the absolute value of a weak solution that spends zero time at $0$. 

\begin{thm}
\label{t:wu}
Let $\alpha\in (-1,1)$, and
let $X$ be a weak solution of \eqref{e:1} started at $X_0$ such that $X$ spends zero time at $0$. Then
the law of the process $Z=\bigl(\frac{1}{1-\alpha}|X_t|^{1-\alpha}\bigr)_{t\geq 0}$ coincides with the law of a reflected Brownian motion 
started at $\frac{1}{1-\alpha}|X_0|^{1-\alpha}$.
\end{thm}

Having the law of $|X|$ in hand we will describe all possible laws of the solution $X$ itself. Essentially we look for a process which
behaves as a Brownian motion outside of $0$ and spends zero time at $0$. 
It is the skew Brownian motion $B^\theta$ which comes to mind first as an example of a process different from $B$ and $|B|$ which satisfies these 
conditions.

The skew Brownian motion is a homogeneous strong Markov process however it is not the unique process whose absolute 
value is distributed like a 
reflected Brownian motion. 

\ansref{2}{Indeed, one can construct} the so-called variably skewed Brownian motion with a 
variable skewness parameter $\theta\colon \bR\to (-1,1)$ as a solution to the SDE 
\ban
B^{\Theta}_t=x+ B_t+\Theta(L_t(B^{\Theta})),\quad t\geq 0,\quad x\in\bR,
\ean
where $\Theta(x) = \int_0^x \theta(y)\di y$. \chng{This process 
with $|B^{\Theta}|\stackrel{\di}{=}|B|$ (see \cite[Lemma 2.1]{BBKM-00}); however, if $\theta$ is non-constant, $B^\Theta$ is not Markov as the skewness parameter depends on the value of local time.

On the other hand, \cite{etore2012existence} showed that the \emph{inhomogeneous skew Brownian motion} which is a unique strong solution 
of the SDE
\ban
B^{\beta}_t=x+  B_t+\int_0^t \beta(s)\,\di L_s(B^{\beta}),\quad t\geq 0,
\ean
with a deterministic Borel function $\beta\colon [0,\infty)\to[-1,1]$, is an inhomogeneous strong Markov process, and $|B^\beta|\stackrel{\di}{=}|B|$
for $x=0$ (see also \cite{weinryb1983etude}).
}

To exclude these processes we restrict ourselves to the case of homogeneous strong Markov solutions.
\chng{
\begin{thm}
\label{t:ws}
Let $\alpha\in (-1,1)$, and
let $(\widetilde X,\widetilde B)$ be a weak solution of \eqref{e:1} such that $\widetilde X$ is a 
homogeneous strong Markov process spending zero time at $0$. Then there exists 
$\theta\in [-1,1]$
such that
\be \label{e:weakbtheta}
\widetilde X =\Big((1-\alpha)\widetilde B^\theta\Big)^\frac{1}{1-\alpha}
\ee
with a $\theta$-skew Brownian motion $\widetilde B^\theta$, which solves
\be
\label{e:tildeB}
\widetilde B^\theta_t =\frac{1}{1-\alpha}(X_0)^{1-\alpha}+ \widetilde B_t + \theta L_t\bigl(\widetilde B^\theta\bigr),\quad t\ge 0.
\ee  
Moreover, 
$\widetilde X$ is also a strong solution to 
\be\label{e:tildesde}
\di\widetilde{X}_t = \bigl|\widetilde{X}_t\bigr|^\alpha \circ \di \widetilde B_t,\quad t\ge 0. 
\ee
\end{thm}
}
For $\theta=1$, the skew Brownian motion $B^1$ is a non-negative reflected Brownian motion. 
\cite{AryasovaP-11} studied non-negative solutions of a singular SDE written in the weak form.  
By \cite[Theorem 1]{AryasovaP-11} there
exists a strong solution to equation \eqref{e:1} (in the weak form) with initial condition $X_0\geq 0$ spending zero time at the
point $0$ and the strong uniqueness holds in the class of solutions spending zero time at $0$. Of course it is equal to
the solution $X^1$ which can be determined explicitly
as
\chng{\ban
X^1_t=\left( (1-\alpha) B_t +  X_0^{1-\alpha} + \Big((1-\alpha)\min_{s\leq t}B_s +   X_0^{1-\alpha}\Big)_- \right)^\frac{1}{1-\alpha},\quad t\geq 0,
\ean
where $x_- = - \min(x,0)$ denotes the negative part of $x$.}

Finally we show the existence of strong solutions different from the benchmark solution \eqref{e:benchmark} and characterize all solutions which are homogeneous  strong Markov processes spending zero time at $0$.

\begin{thm}
\label{t:main}
1. Let $\alpha\in (0,1)$ and $\theta\in [-1,1]$. 
Let $X_0\in\bR$ and let $B^\theta$ be the unique strong solution of the SDE
\ben
B^\theta_t =\frac{1}{1-\alpha}(X_0)^{1-\alpha}+ B_t + \theta L_t\bigl(B^\theta\bigr),\quad t\ge 0.
\een
Then 
\be
\label{e:xtheta}
X^\theta_t=\left((1-\alpha)B^\theta_t\right)^\frac{1}{1-\alpha}
\ee
 is a 
strong solution of \eqref{e:1} which is a homogeneous strong Markov process spending zero time at $0$.

Moreover, $X^\theta$ is the unique
strong solution of \eqref{e:1} which is a homogeneous strong Markov process spending zero time at $0$ and such that
\ban
\P(X^\theta_t\geq 0\mid X_0 = 0)=\beta_+=\frac{1+\theta}{2},\quad t>0.
\ean
2. Let $\alpha\in (-1,0]$. Then the benchmark solution $X^0_t=\big((1-\alpha)B_t+(X_0)^{1-\alpha}\big)^\frac{1}{1-\alpha}$ is the unique
strong solution of \eqref{e:1} which is a homogeneous strong Markov process spending zero time at $0$.
\end{thm}
\begin{rem}
By Theorem~\ref{t:ws}, a similar uniqueness result also holds for weak solutions. 
\end{rem}
\chng[R1: discussion on sub/super-diffusivity added]{
The explicit form of the solutions \eqref{e:xtheta} allows to study their long time behaviour easily. Setting for simplicity $X_0=0$
we recall the transition probability density of the skew Brownian motion (see, e.g.\ \cite[Eq.\ (17)]{Lejay-06}) and find 
the mean square displacement 
\ba
\operatorname{Var} (X_t^\theta) = 
t^{\frac{1}{1-\alpha}}\cdot \bigl(2 (1-\alpha)^2\bigr)^{\frac{1}{1-\alpha}} 
\Big[\pi^{-\frac12} \Gamma\Bigl(\frac{3-\alpha}{2(1-\alpha)}\Bigr) - \theta^2 \Gamma\Bigl(\frac{2-\alpha}{2(1-\alpha)}\Bigr)^2\Big].
\ea
Hence, $X^\theta$ demonstrates the diffusive behaviour $\operatorname{Var} X_t^\theta\sim t$ for $\alpha=0$, as the diffusion coefficient is a.e.\ constant. For
$\alpha\in (0,1)$, the growing diffusion coefficient leads to the superdiffusion; for $\alpha\in (-1,0)$, the diffusion coefficient decreases to zero at infinity, hence we have a subdiffusion. Such behaviour was recovered in \cite{CherstvyCM-13}, where one can find a 
discussion on the physical interpretation.}

The crucial part of the proof of Theorem \ref{t:main} is the existence of the quadratic variation $[|X^\theta|^\alpha,B]$ which follows from the 
following Theorem which is interesting on its own.
\begin{thm}
\label{t:cov}
Let $f\in L^2_{\mathrm{loc}}(\bR)$ and let the $\theta$-skew Brownian motion $B^\theta$, $\theta\in [-1,1]$, 
be the unique strong solution of the SDE \eqref{e:sdesbm}.
Then the quadratic variation
\ben
[f(B^\theta),B]_t = \lim_{n\to\infty} \sum_{t_k\in D_n,t_k< t} 
 \big(f(B^\theta_{t_k})-f(B^\theta_{t_{k-1}})\big)(B_{t_k}-B_{t_{k-1}})
\een
exists as a limit in u.c.p.
Moreover, let $\{h_m\}_{m\geq 1}$ be a sequence of continuous functions such that for each $A>0$
\ba
\lim_{m\to \infty}\int_{-A}^A |h_m(x)-f(x)|^2\,\di x=0.
\ea
Then
\ben
[h_m(B^\theta),B]_t\to  [f(B^\theta),B]_t,\quad m\to\infty,
\een
in u.c.p.
\end{thm}

The proof of this Theorem uses the approach by \cite{FPSh-95}. For $\theta\in(-1,1)\setminus \{0\}$ it is combined with the time reversal technique 
from \cite{haussmann1985time,haussmann1986time}; for $\theta=\pm 1$, we use the time-reversal results for the reflected Brownian motion by \cite{Petit-97}.

%

%
%
%
%
%

\section{On the relation between the Stratonovich and It\^o equations\label{sec:etc}} 

\chng[R2: explanation added]{Recall that a Stratonovich SDE $\di X=f(X)\circ \di B$ with a smooth function $f$ can be rewritten in the It\^o form as
$\di X=f(X)\, \di B +\frac12 f(X)f'(X)\,\di t$ (see, e.g.\ \cite[Chapter 5]{Protter-04}). Although 
for $\alpha\in(-1,1)$ the function $x\mapsto |x|^\alpha$ is not smooth, let us \emph{formally} 
write the Stratonovich SDE \eqref{e:1}  
as an It\^o SDE with irregular/singular coefficients}
\ba
\label{e:ito}
X_t=X_0+\int_0^t |X_s|^\alpha\,\di  B_s+ \frac{\alpha}{2}\int_0^t (X_s)^{2\alpha-1}\,\di s.
\ea
Let us check whether the process $X^\theta$ satisfies this equation. For definiteness, we set $X_0=0$. 

In order to be able to substitute $X^\theta$ into \eqref{e:ito} we have to 
guarantee that the both summands of the SDE \eqref{e:ito} are well defined. Hence, 
for the existence of the It\^o integral we need
\ba
\label{e:ex1}
\int_0^t |X_s^\theta|^{2\alpha}\,\di s\stackrel{\di}{=}\int_0^t |B_s|^\frac{2\alpha}{1-\alpha}\,\di s<\infty\quad \text{a.s.,} 
\ea
and for the existence of the drift term we need
\ba
\label{e:ex2}
\int_0^t |X_s^\theta|^{2\alpha-1}\,\di s\stackrel{\di}{=}\int_0^t |B_s|^\frac{2\alpha-1}{1-\alpha}\,\di s<\infty\quad \text{a.s.} 
\ea

The Engelbert--Schmidt zero-one law \cite[Theorem 1]{EngelbertSchmidt-81} implies that for a Borel function $\Phi\colon \bR\to [0,+\infty]$ 
\ban
\P\Big(\int_0^t \Phi(B_s)\,\di s<\infty,\ \forall \, t\geq 0 \Big)=1\quad \Leftrightarrow\quad \Phi\in L^1_\text{loc}(\bR),
\ean
and hence \eqref{e:ex1} is satisfied for all $\alpha>-1$, and the drift term \eqref{e:ex2}
exists for $\alpha>0$. This indicates that $X^\theta$ is a solution of \eqref{e:ito} for $\theta\in[-1,1]$ and $\alpha\in(0,1)$.

To extend the existence result to $\alpha\in (-1,0]$, we will consider the drift term in the principal value sense: 
\ba
\label{e:vp}
\text{v.p.}\int_0^t (B_s)^\frac{2\alpha-1}{1-\alpha}\,\di s:=\lim_{\e \downarrow 0} 
\int_0^t (B_s)^\frac{2\alpha-1}{1-\alpha}\cdot \bI(|B_s|> \e)\,\di s.
\ea
The principal value definition is intrinsically 
based on the symmetry of the Brownian motion and hence excludes the skew cases $\theta\neq 0$.
Necessary and sufficient conditions for the existence of Brownian principal value integrals are given in
\cite[Theorem 3.1, p.\ 352]{Cherny-01}. In particular, the integral \eqref{e:vp} is finite if and only if $\alpha>-1$.

This yields that for $\alpha\in (-1,0]$, $X^0$ is the solution of the It\^o SDE 
\ba
\label{e:itovp}
X_t=X_0+\int_0^t |X_s|^\alpha\,\di  B_s+ \frac{\alpha}{2}\cdot \text{v.p.}\int_0^t (X_s)^{2\alpha-1}\,\di s.
\ea

\chng[R2: rewritten]{In their book, \cite{CheEng05} consider singular SDEs in the sense of existence of the Lebesgue integrals \eqref{e:ex1} and \eqref{e:ex2}. 
It follows from \cite[Chapter 5]{CheEng05}, that for $\alpha\leq 0$ the SDE \eqref{e:ito} has a unique solution 
which sticks to $0$ after hitting it. This behaviour seems to contradict the fact that the benchmark solution $X^0$ is a solution to the 
Stratonovich equation which spends zero time at $0$ for $\alpha\in (-1,1)$ due to Theorem \ref{t:exists}.
This contradiction is resolved by taking into account the fact that for $\alpha\in(-1,0]$ the 
noise-induced drift has to be understood in the \emph{principal value} sense \eqref{e:vp} and not as a Lebesque integral.}

\chng[R1: formulation added]{These observations lead to the following theorem.
\begin{thm}
1. For $\alpha\in (0,1)$ and $\theta\in[-1,1]$, the process $X^\theta$ given by \eqref{e:xtheta} is a strong solution of the It\^o SDE \eqref{e:ito}.
2.  For $\alpha\in (-1,0]$, the process $X^0$ given by \eqref{e:benchmark} is a strong solution of the It\^o SDE \eqref{e:itovp}.
\end{thm}
}

%
%
%
%

\section{Proof of Theorem \ref{t:wu}} \label{sec:six}

We use the following characterization of the reflected Brownian  motion, see \cite{Varadhan-RBM}. 

\begin{prp}
Let $(\Omega,\rF,\bF,\P)$ be a filtered probability space. 
A continuous non-negative stochastic process $Z$ is a reflected Brownian motion started at $x$ if and only if
\begin{enumerate}
\item $Z_0 = x$ a.s.;
\item 
$Z$ behaves locally like a Brownian motion on $(0,\infty)$, i.e.\ for any
bounded smooth function $f\colon [0,\infty)\to \bR$ such that $f(x)=0$ 
for $x\in[0,\delta]$ for some $\delta=\delta(f)>0$, the process
\ban
f (Z_t) -f(x) -\frac12 \int_0^t f'' (Z_s)\,\di s
\ean
is a martingale;
\item $Z$ spends a zero time at $0$, i.e.
\ban 
\E \int_0^\infty \bI_{\{0\}} (Z_s)\,\di s = 0.
\ean
\end{enumerate}
\end{prp}
For the proof of Theorem \ref{t:wu} we will need a variant of the change of variables formula for functions vanishing on a neighborhood of the 
irregular point of the SDE.
\begin{lem}
\label{l:ito}
Let $\phi\in C^1(\bR\backslash\{0\})$ and let $(\widetilde X,\widetilde B)$ be a weak solution of the SDE
\ban
&X_t=x+\int_0^t \phi(X_s)\circ \di B_s:=x+\int_0^t \phi(X_s)\, \di B_s+\frac12 [\phi(X),B]_t.
\ean
Then for any $g\in C^2(\bR)$ which vanishes on a neighborhood of zero we have  
\ba
g(\widetilde X_t)& =g(X_0)+\int_0^t g'\bigl(\widetilde X_s\bigr)\phi\bigl(\widetilde X_s\bigr) \,\di\widetilde  B_s \\
& {} + 
\frac12 \int_0^t \varphi\bigl(\widetilde X_s\bigr)\Big(g''\bigl(\widetilde X_s\bigr)\phi\bigl(\widetilde X_s\bigr)+g'\bigl(\widetilde X_s\bigr)\phi'\bigl(\widetilde X_s\bigr) \Big) \,\di s.
\ea
\end{lem}
The proof of this Lemma essentially follows the lines of the proof of the classical It\^o
formula for It\^o processes and is postponed to Appendix \ref{a:ito}.

Eventually, we prove Theorem~\ref{t:wu}. Let $\bigl(\widetilde X,\widetilde B\bigr)$ be a weak solution of the SDE \eqref{e:1} spending zero time at $0$.
We consider the process
\ban
Z_t=\frac{1}{1-\alpha}\bigl|\widetilde X_t\bigr|^{1-\alpha},\quad t\geq 0.
\ean
which starts at $Z_0=\frac{1}{1-\alpha}|X_0|^{1-\alpha}$ and also spends zero time at $0$.

Let $f\in C^2_b(\bR_+)$  be zero on a neighborhood of $0$. The function
$g(x)=f\bigl(\frac{1}{1-\alpha}|x|^{1-\alpha}\bigr)$ is also twice continuously differentiable, bounded, 
and is zero on a neighborhood of $0$, and 
\ban
g'(x)&=f'(z)(x)^{-\alpha},\\
g''(x)&=f''(z)|x|^{-2\alpha}-\alpha f'(z)|x|^{-\alpha-1},\quad z= \frac{1}{1-\alpha}|x|^{1-\alpha}.
\ean
Then Lemma~\ref{l:ito} immediately yields
\ban 
f(Z_t)&=g\bigl(\widetilde X_t\bigr)
=g(X_0)+\int_0^t g'\bigl(\widetilde X_s\bigr)\bigl|\widetilde X_s\bigr|^\alpha\,\di\widetilde  B_s\\
&\qquad\qquad\qquad  + 
\frac12 \int_0^t \Big(g''\bigl(\widetilde X_s\bigr)\bigl|\widetilde X_s\bigr|^{2\alpha}+\alpha g'\bigl(\widetilde X_s\bigr)\bigl(\widetilde X_s\bigr)^{2\alpha-1} \Big) \,\di s\\
&=f(Z_0)+\int_0^t f'(Z_s)\sign \widetilde X_s\,\di \widetilde B_s +\frac12 \int_0^t f''(Z_s)\,\di s,
\ean
so that the process
\ba
t\mapsto
f(Z_t)-f(Z_0)-\frac12 \int_0^t f''(Z_s)\,\di s
\ea
is a martingale.

\section{Proof of the  Theorem \ref{t:ws}} \label{sec:seven}

Let $\bigl(\widetilde X,\widetilde B\bigr)$ be a weak solution of the SDE \eqref{e:1} spending zero time at 0. Then, by Theorem~\ref{t:wu},
\ban
\frac{1}{1-\alpha}\bigl|\widetilde X_t\bigr|^{1-\alpha} \stackrel{\di}{=} \Big|W - \frac{X_0}{1-\alpha}  \Big|
\ean
for some standard Brownian motion $W$, i.e.\ is a reflected Brownian motion starting at  $\frac{|X_0|}{1-\alpha}$.
We first establish \eqref{e:weakbtheta}. 

\begin{prp}\label{p:weak}
Let $Y$ be a continuous homogeneous strong Markov process starting at $y\in\bR$ such that $|Y|\stackrel{\di}{=}|W-y|$, $W$ being a standard 
Brownian motion. 
Then there is $\theta\in[-1,1]$ such that
$Y\stackrel{\di}{=}B^\theta$, where $B^\theta$ is the $\theta$-skew Brownian motion starting at $y$.
\end{prp}


\begin{proof}
Since for any $\theta\in[-1,1]$,  $Y\stackrel{\di}{=}W+y\stackrel{\di }{=}B^\theta$ 
before the first hitting time of $0$, it is sufficient to consider the case of the initial starting point $y=0$.

Denote for $a<0<b$
\ban
\tau_{(a,b)}=\inf\{t\geq 0\colon Y_t\notin (a,b)\}
\ean
and show that the probability
\ban
p_+(\e)=\P(Y_{\tau_{(-\e,\e)}}=\e|Y_0=0),\quad \e>0,
\ean
does not depend on $\e$. 

Indeed, if $p_+(\e)= 0$ or $p_+(\e)=1$ for all $\e>0$, then the statement holds true.

Assume that there is $\e>0$ such that $p_+(\e)=\beta_+\in (0,1)$.

\chng[R2: rewritten]{
Let $0<\e<\e'$, then
\ban
p_+(\e')&=\P\Big(Y_{\tau_{(-\e',\e')}}=\e'\,\Big|\, Y_{\tau_{(-\e,\e)}}=\e,Y_0=0\Big)p_+(\e)
\\&\ + \P\Big(Y_{\tau_{(-\e',\e')}}=\e'\,\Big|\,Y_{\tau_{(-\e,\e)}}=-\e,Y_0=0\Big)(1-p_+(\e))\\
&=\P(Y_{\tau_{(-\e',\e')}}=\e'\mid Y_0=\e)p_+(\e)+ \P(Y_{\tau_{(-\e',\e)}}=\e'\mid Y_0=-\e)(1-p_+(\e)).
\ean
Since $\Law(Y_t; 0\leq t\leq \tau(0,\e')\mid Y_0=\e)= \Law (B_t; 0\leq t\leq \tau(0,\e')\mid B_0=\e  )$ we get by virtue of the gambler ruin problem for Brownian motion that
\ban
&\P(Y_{\tau_{(0,\e')}}=\e'\mid Y_0=\e)=\P(B_{\tau_{(0,\e')}}=\e'\mid B_0=\e)=\frac{\e}{\e'},\\
&\P(Y_{\tau_{(0,\e')}}=0\mid Y_0=\e)=\P(B_{\tau_{(0,\e')}}=0\mid B_0=\e)=1-\frac{\e}{\e'},
\ean
and hence 
\ban
\P(Y_{\tau_{(-\e',\e')}}=\e'\mid Y_0=\e)&=\P(Y_{\tau_{(0,\e')}}=\e'\mid Y_0=\e)+\P(Y_{\tau_{(0,\e')}}=0\mid Y_0=\e)p_+(\e')\\
&=\frac{\e}{\e'}+\Big(1-\frac{\e}{\e'}\Big)p_+(\e').
\ean
Analogously
\ban
\P(Y_{\tau_{(-\e',\e')}}=\e'\mid Y_0=-\e)&= \P(Y_{\tau_{(-\e',0)}}=0\mid Y_0=-\e)p_+(\e')=\Big(1-\frac{\e}{\e'}\Big)p_+(\e').
\ean
}
Hence we obtain that
\ban
p_+(\e')&=\frac{\e}{\e'}p_+(\e)+\Big(1-\frac{\e}{\e'}\Big)p_+(\e')p_+(\e)+  \Big(1-\frac{\e}{\e'}\Big)p_+(\e')   (1-p_+(\e))=p_+(\e)=\beta_+.
\ean
Let now $0<\e'<\e$. Due to the continuity of the paths of $Y$, $p_+(\e')>0$, so repeating the previous argument with $\e$ and $\e'$ interchanged
we eventually obtain that $p_+(\e)=\beta_+$ for all $\e>0$.

Since $Y$ is a continuous strong Markov process its law is uniquely determined by the Dynkin characteristic operator
\ban
\mathfrak A f(x):=\lim_{U\downarrow x}\frac{\E_x f(Y_{\tau(U)})-f(x)}{\E_x\tau(U)} ,
\ean
where $U$ is a bounded open interval containing $x$, see \cite[Chapter 5 \S 3]{Dynkin65}.

Choosing $U=U_\e=(x-\e,x+\e)$, a straightforward calculation yields that for $x\neq 0$ and $f$ being twice continuously differentiable at $x$
\ban
\mathfrak A f(x)=\lim_{\e\downarrow 0}\frac{\frac12 f(x+\e)+\frac12f(x-\e) -f(x)}{\e^2}=\frac12f''(x).
\ean
For $x=0$ the limit
\ban
\mathfrak A f(0)=\lim_{\e\downarrow 0}\frac{\beta_+ f(\e)+(1-\beta_+)f(-\e) -f(0)}{\e^2}=\frac12\Big(\beta_+ f''(0+)+(1-\beta_+)f''(0-)\Big)
\ean
exists for any continuous $f$ such that $\beta_+ f'(0+)=(1-\beta_+)f'(0-)$ and $f''(0+)$ and $f''(0-)$ exist and $f''(0+)=f''(0-)$.

Hence $\mathfrak A$ coincides with the generator of the $\theta$-skew Brownian motion with $\theta=2\beta_+-1$ 
(see \cite{Lejay-06}).
\end{proof}
\ansref{2}{
By Proposition~\ref{p:weak} and Theorem~\ref{t:wu},
the process $\widetilde B^\theta =\frac{1}{1-\alpha}\bigl(\widetilde X\bigr)^{1-\alpha}$
is a $\theta$-skew Brownian motion with some $\theta\in[-1,1]$, starting at $B^\theta_0=\frac{1}{1-\alpha}\bigl(X_0\bigr)^{1-\alpha}$. Equivalently,
\be\label{e:wB}
\widehat B_t =\frac{1}{1-\alpha}\bigl(X_0\bigr)^{1-\alpha}+ \widetilde B_t^\theta - \theta L_t(\widetilde B^\theta) 
\ee
is a standard Wiener process. Comparing with \eqref{e:tildeB}, we need to show that $\widehat B = \widetilde B$. By the results of \cite{HShepp-81}, $\widetilde B$ is adapted to the filtration generated by $\widehat B$. Now we want to show that $\widehat B = \widetilde B$ a.s. By Lemma~\ref{l:ito}, for any function $g\in C^2(\bR)$, vanishing on a neighbourhood of $0$, and for any $t>0$ it holds
\ban
g\bigl(\widetilde X_t\bigr)& =g(X_0)+\int_0^t g'\bigl(\widetilde X_s\bigr)\bigl|\widetilde X_s\bigr|^{\alpha} \,\di\widetilde  B_s\\
&\qquad  + 
\frac12 \int_0^t \bigl|\widetilde X_s\bigr|^{\alpha}\Bigl(g''\bigl(\widetilde X_s\bigr)\bigl|\widetilde X_s\bigr|^{\alpha}+\alpha g'\bigl(\widetilde X_s\bigr)\bigl|\widetilde X_s\bigr|^{\alpha-1}\Bigr) \,\di s.
\ean
A similar formula, but with $\widehat B$ in place of $\widetilde B$, holds thanks to the equality $\widetilde X=\big((1-\alpha)\widetilde B^\theta\big)^{\frac{1}{1-\alpha}}$, \eqref{e:wB} and the usual It\^o formula for semimartingales applied to $\widetilde B^\theta$. Indeed, 
}
\chng[R2: explanations added]{
\ban
g\bigl(\widetilde X_t\bigr)&= g\Big(\big((1-\alpha)\widetilde B^\theta_t\big)^{\frac{1}{1-\alpha}}\Big)\\
&=g(X_0)+\int_0^t g'\bigl(\widetilde X_s\bigr)\bigl|\widetilde X_s\bigr|^{\alpha} \,\di\widehat  B_s
+\theta \int_0^t g'\bigl(\widetilde X_s\bigr)\bigl|\widetilde X_s\bigr|^{\alpha} \,\di L_s(\widetilde B^\theta)
\\
&\qquad  + 
\frac12 \int_0^t \bigl|\widetilde X_s\bigr|^{\alpha}\Bigl(g''\bigl(\widetilde X_s\bigr)\bigl|\widetilde X_s\bigr|^{\alpha}+\alpha g'\bigl(\widetilde X_s\bigr)\bigl|\widetilde X_s\bigr|^{\alpha-1}\Bigr) \,\di s\\
&=g(X_0)+\int_0^t g'\bigl(\widetilde X_s\bigr)\bigl|\widetilde X_s\bigr|^{\alpha} \,\di\widehat  B_s\\
&\qquad  + 
\frac12 \int_0^t \bigl|\widetilde X_s\bigr|^{\alpha}\Bigl(g''\bigl(\widetilde X_s\bigr)\bigl|\widetilde X_s\bigr|^{\alpha}+\alpha g'\bigl(\widetilde X_s\bigr)\bigl|\widetilde X_s\bigr|^{\alpha-1}\Bigr) \,\di s\\
\ean
}
Consequently,
$$
\int_0^t g'\bigl(\widetilde X_s\bigr)\bigl|\widetilde X_s\bigr|^{\alpha} \,\di\widetilde  B_s = \int_0^t g'\bigl(\widetilde X_s\bigr)\bigl|\widetilde X_s\bigr|^{\alpha} \,\di\widehat  B_s. 
$$
Now taking a sequence of non-negative functions $g_n\in C^2(\bR)$, $n\ge1$, vanishing in some neighborhood of $0$ and such that $g_n\ge 0$ $g_n'(x)|x|^\alpha \uparrow 1$, $x\neq 0$, $n\to\infty$, we get
$$
\int_0^t \bI_{\widetilde X_t\neq 0} \di\widetilde B_t = \int_0^t \bI_{\widetilde X_t\neq 0} \di\widehat B_t
$$
a.s. Since $\bI_{\widetilde X_t\neq 0} = 1$ a.e.\ by assumption, it follows that $\widetilde B_t = \widehat B_t$ a.s. As a result, $\widetilde B^\theta$ is adapted to the augmented filtration of $\widetilde B$, so in view of \eqref{e:weakbtheta}, the same is true for $\widetilde X$. Since $\widetilde X$ also satisfies \eqref{e:tildesde} by definition, it is a strong solution.

\section{Proof of Theorem \ref{t:cov}\label{s:Y}} 

Let $\theta\in(-1,1)\setminus\{0\}$; the case of $\theta=\pm 1$ is considered in Appendix \ref{a:petit}, and the case $\theta=0$ is covered by \cite{FPSh-95}.
Define
\ban
\sigma(x) = \frac{2}{1+\theta \sign x} 
\quad \text{ and }\quad 
\beta(x) = \frac{1}{\sigma(x)}.
\ean 
Let $Y^\theta$ be the unique strong solution of the SDE 
\ba
\label{e:Y}
Y_t^\theta=u+\int_0^t \sigma(Y_s^\theta)\,\di B_s,\quad u\in \bR,
\ea
and consider the two-dimensional Markov process $(Y^\theta,B)$ with the law 
$$
\P_{u,w}:=\Law((Y^\theta,B)|Y_0^\theta = u,\,B_0 = w).
$$ 
The process $Y^\theta$ is called an oscillating Brownian motion, see e.g.\ \cite{KeilsonW-78,LejPig18}. 
\chng[R1: reference added]{}

The skew Brownian motion with parameter $\theta$, starting from $w_0\in\bR$ and driven by a Brownian motion $B$ 
is the unique strong solution to the following stochastic differential equation
\begin{equation}\label{eq:skewbm}
B_t^\theta  = w_0 + (B_t-w) + \theta L_t(B^\theta).
\end{equation}

Further, define the functions
\ban
r(x)&=\frac{x}{\sigma (x)}=x\beta(x)= 
\begin{cases}
\frac{1+\theta}{2}x, & x\geq 0,\\
\frac{1-\theta}{2}x, & x<0,
\end{cases} \qquad s(x) = x\sigma(x)=\frac{x}{\beta(x)},
\ean 
then $s(r(x)) \equiv x$. 
The application of the It\^o--Tanaka formula (compare with \cite[Section 5.2]{Lejay-06}) yields 
\ba
\label{eq:dep}
r(Y_t^\theta)& =r(u)+B_t-w+ \frac{\theta}{2} L_t(Y_t^\theta)\\
&= r(u)+B_t-w+ \theta L_t(B^\theta)=r(u)-w_0+B^\theta_t.
\ea
%
%

In the following lemma we will use the functional dependence \eqref{eq:dep} of the processes $(Y^\theta,B)$ and $(B^\theta, L(B^\theta))$ to 
determine the marginal density 
of the pair $(Y^\theta_t,B_t)$.

\begin{lem}
For $\theta\in(-1,1)\setminus\{0\}$, $t>0$, the joint distribution of $Y_t^\theta$ and $B_t$ given $Y_0^\theta = u$, $B_0 = w$ is 
\ba\label{eq:YBdensx}
\P_{u,w}(Y_t^\theta\in \di y, B_t\in \di z)  & =  
\frac{2\beta^2(y)}{\theta^2 \sqrt{2\pi t^3}}
\Big(2y\beta^2(y)
- \kappa u
- z +w\Big)\\
&\ \times \exp\Big(-\frac{1}{2\theta^2 t}\Big(  2y  \beta^2(y) - \kappa u
- z +w )^2\Big)\Big)\,\di z\, \di y,
\ea
where $\kappa =  \frac{1}{2}(1-\theta^2)$, if $\theta^{-1} (r(y) - r(u) - z  +w )> 0$,
and 
\ba
\label{eq:YBdistx}
&\P_{u,w} \Big(Y_t^\theta\in \di y, B_t
= w  +r(y) - r(u)\Big) \\&\quad 
=\frac{\beta(u) }{|\theta|\sqrt{2\pi t}}\Big(\ex^{-\frac{(r(y)-r(u))^2}{2t}} - \ex^{-\frac{(r(y)+r(u))^2}{2t}}  \Big)\cdot \bI_{uy>0}\, \di y.
\ea
In particular, the joint density of $(Y^\theta_t,B_t)$ provided that $Y_0=u=0$, $B_0 =w=0$ is 
\begin{equation}
\label{eq:jointYBdensity}
p(t,y,z) =  \frac{2 \beta^2(y)}{\theta^2\sqrt{2\pi t^3}}
\Big(2y\beta^2(y) - z\Big) \exp\Big(-\frac{1}{2\theta^2 t}\Big(  2y\beta^2(y) 
- z \Big)^2\Big)\cdot  \bI_{\theta^{-1}(r(y) - z) >0 }.
\end{equation}
\end{lem}
\begin{proof}
\chng[R1: references added, the proof shortend]{
The joint distribution of $B^\theta_t$ and $L_t(B^\theta)$, $t>0$, is well known and can be found, e.g.\ in 
\cite{Appuhamillage2011corrections,etore2012existence,gairat2017density}:
%
\ba
\label{eq:BLdensity}
&\P\Big(B_t^\theta\in \di b, L_t(B^\theta)\in \di l\Big| B_0^\theta = w_0\Big)
\\&= 2\beta_+\cdot \bI_{b\geq 0}\frac{l+ |w_0|+\abs{b} }{\sqrt{2\pi t^3}} \ex^{-\frac{(l+ |w_0|+\abs{b})^2}{2t}}\di z\, \di l\\
&\ +2\beta_-\cdot \bI_{b<0}\frac{l+ |w_0|+|b| }{\sqrt{2\pi t^3}}  \ex^{-\frac{(l+ |w_0|+\abs{b})^2}{2t}}\di z\, \di l\\
&=2\beta(b)\cdot \frac{l+ |w_0|+|b| }{\sqrt{2\pi t^3}} \ex^{-\frac{(l+ |w_0|+|b|)^2}{2t}}\di z\, \di l,\quad  b\in \bR,\ l>0,
\ea
where in the last equality we redefined the value of the density at $b=0$ for convenience.
}

Recall now that
\ban
&B^\theta_t=r(Y^\theta_t)-r(u)+w_0,\\
&L_t(B^\theta)=\frac1\theta\Big(r(Y^\theta_t)-r(u)-B_t +w \Big).
\ean
The initial condition of $u=Y^\theta_0$ given, let us fix $w_0=r(u)$, so that $B^\theta_t=r(Y^\theta_t)$ for $t\geq 0$.
Then the change of variables $b=b(y,z)$, $l=l(y,z)$,
\ba
\label{eq:chofvars}
&b=r(y),\quad l=\frac1\theta\Big(r(y)-r(u)-z +w\Big),
\ea
yields
\ba
\label{eq:l+x_0+b}
l+ |w_0|+|b|&=\frac1\theta\Big(r(y)-r(u)-z +w \Big) +|r(u)|+|r(y)|\\
&= \frac{ r(y)+\theta |r(y)|- r(u)+ \theta |r(u)|- z + w}{\theta}\\
&= \frac{ r(y)(1+\theta \sign y)- r(u)(1-\theta \sign u)- z + w}{\theta}\\
&= \frac{2y\beta^2(y)}{\theta}
- \frac{u(1-\theta^2)}{2\theta}
- \frac{z - w}{\theta},
\ea
where we made use of the relation $|r(y)|=r(y)\sign y$.
For $y,u\neq 0$, the Jacobian for the change of variables $(y,z)\to (b,l)$ is given by \eqref{eq:chofvars} and its determinant equal
\ban
J=
\begin{pmatrix}
\beta(y) & 0 \\
 \frac{\beta(y)}{\theta} & -\frac{1}{\theta}
\end{pmatrix},
\qquad |\det J| =\frac{\beta(y)}{|\theta|},
\ean
whence, noting that $\sigma(b) = \sigma(y)$, we get \eqref{eq:YBdensx}. 
Similarly, we have \eqref{eq:YBdistx}.
The remaining formula \eqref{eq:jointYBdensity} follows by plugging in
$w=0$.
\end{proof}

From now on we assume without loss of generality that $\theta\in(0,1)$. 
Let all the processes under consideration will be started at zero, $u=w=w_0=0$, so that $Y^\theta= B^\theta/\beta(B^\theta)$.

Note that for any $t>0$
\ba 
\label{2Y_t - B_t}
r(Y_t^\theta) - B_t &  = B_t^\theta - B_t = \theta L_t(B^\theta) > 0,\\
2Y_t^\theta\beta^2(Y_t^\theta) - B_t  & = 2\beta(B_t^\theta) B_t^\theta - B_t = \big(1 + \theta \sign B_t^\theta \big)B_t^\theta - B_t\\ 
& = B_t^\theta - B_t + \theta B_t^\theta \sign B_t^\theta = \theta L_t(B^\theta) + \theta \abs{B_t^\theta} >0. 
\ea

Our aim now is to prove a generalized It\^o formula for $B^\theta$, in the spirit of \cite{FPSh-95}. 
Towards this end, on a fixed time interval $[0,T]$ we first establish a stochastic differential equation for the time-reversed pair 
$(\bar Y_t^\theta, \bar B_t) = (Y_{T-t}^\theta,  B_{T-t})$, which is interesting by its own.
We follow the method developed by \cite{haussmann1985time,haussmann1986time} for Markovian diffusions.  
For $y\neq 0$ and $z < r(y)$, $s\in [0,T)$ define the functions
\ban
\bar b^y (s,y,z) & = \frac{1}{p(T-s,y,z)}\Big(\sigma^2(y)\cdot \frac{\partial p}{\partial y}(T-s,y,z) 
+ \sigma(y)\cdot\frac{\partial p}{\partial z}(T-s,y,z)\Big),\\
\bar b^z (s,y,z) & =\frac{1}{p(T-s,y,z)}\Big(\sigma(y)\cdot \frac{\partial p}{\partial y}(T-s,y,z) 
+ \frac{\partial p}{\partial z}(T-s,y,z)\Big)  = \frac{\bar b^y (s,y,z)}{\sigma(y)},
\ean
and set $\bar b^y (s,0,z) = \bar b^z (s,0,z) = \bar b^y(T,y,z) =  \bar b^z(T,y,z) =  0$. 
Noting that 
\ban
\frac{\partial p}{\partial z}(T-s,y,z)  = 
\frac{2\beta^2(y) }{\theta^2\sqrt{2\pi (T-s)^3}} \exp\Big(-\frac{(2y\beta^2(y)-z)^2}{2\theta^2 (T-s)}\Big) 
\Big( -1 + \frac{\big(2y\beta^2(y)-z\big)^2}{\theta^2 (T-s)}\Big), \ z<r(y),
\ean
and $\frac{\partial p}{\partial z}(T-s,y,z) = - 2 \beta^2(y) \frac{\partial p}{\partial y}(T-s,y,z)$, we get 
\ba
\label{eq:b-y}
\bar b^y (s,y,z) & =\frac{2\beta(y)-1}{\beta(y)}
\Big(\frac{1}{2y\beta^2(y)-z} - \frac{2y\beta^2(y)-z}{\theta^2(T-s)}\Big) \\
&= \frac{\theta\sign y}{ \beta(y)}\cdot \Big(\frac{1}{2y\beta^2(y)-z} - \frac{2y\beta^2(y)-z}{\theta^2(T-s)}\Big) ,\quad z<r(y).
\ea
\begin{prp}
Let for $\theta\in(-1,1)\setminus\{0\}$, $B^\theta$ be a solution of \eqref{eq:skewbm} started at $0$. 
Then for any $T>0$, $(\bar Y^\theta_t, \bar B_t) = (Y^\theta_{T-t},B_{T-t})$ is a weak solution to the stochastic differential equation
\ba 
\label{eq:reversed sde}
\bar Y_t^\theta & = Y^\theta_T + \int_0^t \bar b^y (s,\bar Y^\theta_s, \bar B_s)\, \di s  + \int_0^t \sigma(\bar Y^\theta_s)\,\di \bar W_s,\\
\bar B_t & = B_T + \int_0^t \bar b^z (s,\bar Y^\theta_s, \bar B_s)\, \di s  + \bar W_t, \quad t\in [0,T],
\ea
$\bar W$ being a standard Brownian motion.
\end{prp}
\begin{rem}
a) Thanks to \eqref{2Y_t - B_t}, the coefficients of \eqref{eq:reversed sde} are well defined. 

\noindent
b) Equation \eqref{eq:reversed sde} is very degenerate, and $\bar Y^\theta$ and $\bar B$ 
evolve proportionally whenever $\bar Y^\theta_t\neq 0$. This, however, will not hinder our analysis. 
\end{rem}

\begin{proof}
As above, we continue to assume without loss of generality that $\theta\in (0,1)$ and $T=1$.
We need to show that $(\bar Y^\theta, \bar B)$ is a solution to the martingale problem with the generator 
\ban
\bar{\mathcal L}_t f (y,z) = \Big(\bar b^y(t,y,z)\cdot \frac{\partial}{\partial y}  + \bar b^y(t,y,z)\cdot \frac{\partial}{\partial y} 
+ \frac{\sigma(y)^2}{2}\cdot \frac{\partial^2}{\partial y^2} + \sigma(y)\cdot \frac{\partial^2}{\partial y\,\partial z} 
+ \frac{1}{2}\cdot \frac{\partial^2}{\partial z^2}\Big) f(y,z).
\ean
Thanks to \eqref{2Y_t - B_t}, it is enough to establish
\begin{equation*}
\E \Big[ \Big(f(\bar{Y}^\theta_t, \bar{B}_t) - f(\bar{Y}^\theta_s, \bar{B}_s) 
- \int_s^t \bar{\mathcal L}_u f(\bar{Y}^\theta_u, \bar{B}_u)\,\di u \Big)\cdot  g(\bar{Y}^\theta_s, \bar{B}_s)\Big] = 0
\end{equation*}
for any $0\le s<t<1$ and functions $f,g\in C^\infty(\bR^2)$ having compact support inside the domain $D:=\{(y,z)\in \bR^2\colon r(y)-z>0\}$. 
Equivalently, 
\begin{equation*}
\E \Big[ \Big(f(Y_t^\theta, B_t) - f(Y_s^\theta,B_s) + \int_s^t \bar{\mathcal L}_{T-u} f(Y^\theta_u, B_u)\,\di u \Big)\cdot g(Y^\theta_t, B_t)\Big] = 0
\end{equation*}
for any fixed $0< s<t\le 1$. Define for $(y,z)\in \bR^2$ 
\ban
v(s,y,z) = \E \Big[ g(Y_t^\theta, B_t)\Big| (Y^\theta_s, B_s)= (y,z)\Big].
\ean
It is proved in Appendix \ref{a:pde} that $v$ solves the partial differential equation
\ban
\Big(\frac{\partial}{\partial s} + \mathcal L\Big) v(s,y,z) = 0,
\ean
where
\begin{equation}\label{e:ell}
\mathcal L f(y,z) = \Big(\frac{\sigma(y)^2}{2}\cdot \frac{\partial^2}{\partial y^2} 
+ \sigma(y)\cdot \frac{\partial^2}{\partial y\,\partial z} + \frac{1}{2}\cdot \frac{\partial^2}{\partial z^2} \Big) f(y,z).
\end{equation}
Denote by $\scprod{\cdot,\cdot}$ the scalar product in $L^2(\bR^2)$. Write
\begin{align*}
 \E[f(Y^\theta_t,B_t)g(Y^\theta_t,B_t)] & - \E[f(Y^\theta_s,B_s)g(Y^\theta_t,B_t)] \\
&=  \E[f(Y^\theta_t,B_t)v(t,Y^\theta_t,B_t)]  - \E[f(Y^\theta_s,B_s)v(s,Y^\theta_s,B_s)]\\
& = \scprod{fp(t),v(t)}  - \scprod{fp(s),v(s)}  \\
&=\int_s^t \scprod{f \frac{\partial}{\partial u} p(u),v(u)}\,\di u + \int_s^t \scprod{f p(u),\frac{\partial }{\partial u} v(u)}\,\di u \\
& = \int_s^t \scprod{f\frac{\partial}{\partial u} p(u),v(u)}\,\di u - \int_s^t \scprod{f p(u),\mathcal L v(u)}\,\di u.
\end{align*}
Using \eqref{eq:Libp} from the Appendix \ref{a:pde}, we get
\begin{align*}
\scprod{fp(u),\mathcal L v(u)\vphantom{\big)}} = & \scprod{v(u),\mathcal L \big(fp(u)\big)}.
\end{align*}
Write 
\begin{equation}\label{eq:p(u,y,z)}
p(u,y,z) = \frac{2\beta^2(y)}{\theta} \phi'_{u}\Big(\frac{z-2y\beta^2(y)}{\theta}\Big)\bI_{z<r(y)},
\end{equation}
where $\phi_t(x) = \frac{1}{\sqrt{2\pi t}} \ex^{-\frac{x^2}{2t}}$ is the standard Gaussian density.
Now we have
\begin{align*}
\mathcal L \big(f p(u)\big)(y,z) = &f(y,z)\cdot \mathcal L p(u,y,z) + p(u,y,z)\cdot \mathcal Lf(y,z)\\
& + \Big(\sigma(y) \cdot \frac{\partial}{\partial y} f(y,z) 
+ \frac{\partial}{\partial z} f(y,z)\Big)\Big(\sigma(y)\cdot \frac{\partial}{\partial y}p(u,y,z) + \frac{\partial}{\partial z}p(u,y,z)\Big).
\end{align*}
From \eqref{eq:p(u,y,z)}, for $y\neq 0$, $z<r(y)$,
\begin{align*}
\mathcal L p(u,y,z)&  = \Big(\frac{4\beta^4(y)}{\theta^3} - \frac{4\beta^3(y)}{\theta^3} + \frac{\beta^2(y)}{\theta^3}\Big)\cdot
\phi'''_u\Big(\frac{z-2y\beta^2(y)}{\theta}\Big) \\
& = \frac{\beta^2(y)}{\theta^3}\cdot ( 2\beta(y)-1 )^2\cdot \phi'''_u\Big(\frac{z-2y\beta^2(y)}{\theta}\Big) 
= \frac{\beta^{2}(y)}{\theta}\cdot \phi'''_u\Big(\frac{z-2y\beta^2(y)}{\theta}\Big).
\end{align*}
On the other hand, since $\frac{\partial}{\partial t} \varphi'_t = \frac{1}{2}\varphi'''_t$, we get
$$
\mathcal L p(u,y,z) = \frac{\partial}{\partial u} p(u,y,z). 
$$
Further, denote
\ban
h(u,y,z) = \sigma(y)\cdot \frac{\partial}{\partial y}p(u,y,z) + \frac{\partial}{\partial z}p(u,y,z). 
\ean
Then we have 
\ban
\scprod{v(u),\mathcal L \big(fp(u)\big)} 
= & \scprod{v(u), f\frac{\partial}{\partial u} p(u)} + \scprod{v(u),p(u) \mathcal L f} 
+ \scprod{v(u),h(u)\Big(\sigma(y)\cdot \frac{\partial}{\partial y}f + \frac{\partial}{\partial z}f\Big)}.
\ean
Observe that 
\ban
\scprod{v(u),p(u) \mathcal L f} & = \int_{-\infty}^\infty \int_{-\infty}^\infty v(u,y,z) p(u,y,z)\mathcal L f (y,z)\di z\, \di y \\
& = \E \big[ v(u,Y^\theta_u,B_u) \mathcal L f(Y_u^\theta,B_u)\big] 
 = \E \big[ \E\big[g(Y_t^\theta,B_t)\big| Y^\theta_u,B_u\big] \mathcal L f(Y^\theta_u,B_u)\big]\\
& =  \E \left[ g(Y^\theta_t,B_t) \mathcal L f(Y^\theta_u,B_u)\right]
\ean
and similarly
\begin{align*}
&\scprod{v(u),h(u)\cdot\Big(\sigma(y)\cdot \frac{\partial}{\partial y}f + \frac{\partial}{\partial z}f\Big)} 
= \scprod{v(u),p(u)\cdot \frac{h(u)}{p(u)}\cdot\Big( \sigma(y)\cdot \frac{\partial}{\partial y}f + \frac{\partial}{\partial z}f\Big)}\\
& = \scprod{v(u),p(u)\Big(\bar b^y (T-u) \frac{\partial}{\partial y}f + \bar b^z (T-u)\frac{\partial}{\partial z}f\Big)} \\
& = \E\Big[g(Y_t^\theta,B_t)\Big(\bar b^y (T-u,Y_u,B_u) \frac{\partial}{\partial y}f(Y^\theta_u,B_u) 
+ \bar b^z (T-u,Y_u,B_u)\frac{\partial}{\partial z}f(Y_u^\theta,B_u)\Big) \Big],
\end{align*}
Collecting everything, 
\begin{align*}
&   \E[f(Y_t^\theta,B_t)g(Y_t^\theta,B_t)] - \E[f(Y_s^\theta,B_s)g(Y^\theta_t,B_t)] 
=   - \E \Big[  g(Y^\theta_t,B_t)\int_s^t \mathcal L f(Y^\theta_u,B_u)\Big] \\
& - \E\Big[g(Y_t^\theta,B_t)\int_s^t\Big(\bar b^y (T-u,Y^\theta_u,B_u) \frac{\partial}{\partial y}f(Y^\theta_u,B_u) 
+ \bar b^z (T-u,Y^\theta_u,B_u)\frac{\partial}{\partial z}f(Y^\theta_u,B_u)\Big)\, \di u\Big]\\
& = - \E\Big[g(Y_t^\theta,B_t)\int_s^t \bar{\mathcal L}_{T-u} f(Y^\theta_u,B_u)\, \di u\Big],
\end{align*}
as required. 
\end{proof}

Consider a sequence of partitions $D_n$ of the form $0=t_0^n<t_1^n<\cdots<t_n^n=T$ with $|D_n|=\max_{1\leq k\leq n}|t_k^n-t^n_{k-1}| \to 0$, $n\to\infty$
(we will often omit the superscript $n$). 

\begin{proof}[Proof of Theorem~\ref{t:cov}]
Note that $f(B^\theta_t) = f(r(Y_t^\theta)) = g(Y_t^\theta)$ with $g\in L^2_{\mathrm{loc}}(\bR)$, so it suffices 
to establish a similar statement for $Y^\theta$. The rest of proof goes similarly to \cite{FPSh-95}. 

First note that by usual localization argument, we can assume that $g\in L^2(\bR)$. Also for a continuous function $h$, the quadratic variation 
$$
 [h(Y^\theta),B]_t = \lim_{n\to\infty} \sum_{t_k\in D_n,t_k<t} \Big(h(Y^\theta_{t_k})-h(Y^\theta_{t_{k-1}})\Big)(B_{t_k}-B_{t_{k-1}})
$$
exists as a limit in u.c.p. Indeed, since $B$ is a semimartingale, and $h(Y^\theta)$ is an adapted continuous process, then 
by \cite[Theorem 21, p.\ 64]{Protter-04},
$$
\lim_{n\to\infty} \sum_{t_k\in D_n,t_k<t}  h(Y^\theta_{t_{k-1}})(B_{t_k}-B_{t_{k-1}}) = \int_0^t h(Y^\theta_s)\,\di B_s 
$$
in u.c.p. The time-reversed process is also a semimartingale, so arguing as in \cite{FPSh-95}, we have the convergence to the backward integral
$$
\lim_{n\to\infty} \sum_{t_k\in D_n,t_k<t}  h(Y^\theta_{t_{k}})(B_{t_k}-B_{t_{k-1}}) = \int_0^t h(Y^\theta_s)\,\di^* B_s
$$
in u.c.p. Therefore, we obtain
$$
[h(Y^\theta),B]_t = \int_0^t h(Y_s^\theta)\,\di^* B_s - \int_0^t h(Y_s^\theta)\,\di B_s.
$$
Fix  some $T>0$ and let now $\{h_m\}_{m\geq 1}$ be a sequence of continuous functions, such that $h_m\to g$ in $L^2(\bR)$ as $m\to\infty$. 
Denote 
\begin{align*}
I_m(t)  &:=  \int_0^t \big(h_m(Y_s^\theta) - g(Y_s^\theta)\big)\,\di B_s,\\
S_{n,m}(t) & := \sum_{t_k\in D_n, t_k\leq t} \big(h_m(Y^\theta_{t_{k-1}}) - g(Y^\theta_{t_{k-1}})\big) (B_{t_k}-B_{t_{k-1}}).
\end{align*}
Since $B^\theta$ is a skew Brownian motion starting from $0$, its density is $\P (B_t^\theta \in \di b) = \frac{\beta(b)}{\sqrt{2\pi t}} \ex^{-\frac{b^2}{2t}}\di b$, $b\in \bR$. 
Then the density of $Y_t^\theta$ satisfies $p(t,y) = \frac{\beta^2(y)}{\sqrt{2\pi t}}\ex^{-\frac{r^2(y)}{2t}}\leq \frac{1}{\sqrt{t}}$, 
so we can estimate, using the Doob inequality that
\begin{align*}
 \E  \sup_{t\in[0,T]}I_m^2(t)  & \leq 4 \int_0^T \E\Big[ \big(h_m(Y_t^\theta) - g(Y_t^\theta)\big)^2 \Big]\,\di t \\
&  = 4\int_0^T \int_{-\infty}^\infty \big(h_m(y) - g(y)\big)^2 p(t,y)\,\di y\, \di t\\
& \le 4\|h_m - g\|^2_{L^2(\bR)}\cdot \int_0^T \frac{\di t}{\sqrt{t}}\le C \|h_m - g\|^2_{L^2(\bR)}.
\end{align*}
Similarly, 
\begin{align*}
\E \sup_{t\in[0,T]} S_{n,m}^2(t) 
 & =   \E \sum_{t_k\in D_n} \big(h_m(Y^\theta_{t_{k-1}}) - g(Y^\theta_{t_{k-1}})\big)^2(t_k-t_{k-1})   \\
& \le C \sum_{t_k\in D_n,k>1} \frac{t_{k}-t_{k-1}}{\sqrt{t_{k-1}}}\|h_m-g\|^2_{L^2(\bR)},
\end{align*}
whence
\begin{align*}
\limsup_{n\to \infty} \E \sup_{t\in[0,T]} S^2_{n,m}(t) 
\le   C \|h_m - g\|^2_{L^2(\bR)}.
\end{align*}
As a result, we get $\sup_{t\in[0,T]}|I_m(t)|\overset{\P}{\longrightarrow} 0$, $m\to \infty$, and for any $\e>0$
$$
\lim_{m\to\infty}\limsup_{n\to \infty}\P \Big(\sup_{t\in[0,T]}\big|S_{n,m}(t)\big|>\e \Big)= 0. 
$$
Hence, using that 
$$
\sum_{t_k\in D_n, t_k<t} h_m(Y^\theta_{t_{k-1}})(B_{t_k}-B_{t_{k-1}})\to 
\int_0^t h_m(Y^\theta_s)\,\di B_s, \quad n\to\infty,
$$
uniformly on $[0,T]$ in probability, we get that 
$$
\sum_{t_k\in D_n, t_k<t} g(Y^\theta_{t_{k-1}})(B_{t_k}-B_{t_{k-1}})\to 
\int_0^t g(Y_s^\theta)\,\di B_s, \quad n\to\infty,
$$
uniformly on $[0,T]$ in probability. 

Further, recall that the time-reversed process $(\bar{Y}^\theta_t,\bar{B}_t) = (Y^\theta_{T-t},B_{T-t})$ satisfies 
\eqref{eq:reversed sde} in the weak sense. As far as the convergence in probability is concerned, we can safely 
assume that $(\bar{Y}^\theta, \bar{B})$ satisfies \eqref{eq:reversed sde} with the same Brownian motion $\bar W$. Then we can write 
\begin{align*}
&\int_0^t g(Y_s^\theta)\, \di^* B_s  =  \int_{T-t}^T g(\bar Y^\theta_s)\,\di \bar{B}_s =
\int_{T-t}^T g(\bar{Y}^\theta_s)\, \di \bar W_s 
+ \int_{T-t}^T g(\bar{Y}^\theta_s)\bar b^y(s,\bar{Y}^\theta_s,\bar B_s)\, \di s.
\end{align*}
Arguing as above, we have 
\begin{align*}
& \sum_{t_k\in D_n, t_k<t} g(Y_{t_{k}}^\theta)(\bar W_{T-t_k}-\bar W_{T-t_{k-1}}) \\
& = \sum_{t_k\in D_n, t_k<t} g(\bar Y^\theta_{T-t_{k}})(\bar W_{T-t_k}-\bar W_{T-t_{k-1}})\to 
\int_{T-t}^T g(\bar Y^\theta_s)\,\di \bar W_s,\ n\to\infty,
\end{align*}
uniformly on $[0,T]$ in probability. 
It remains to show that 
\begin{equation}\label{eq:sumgb->intgb}
\begin{aligned}
& \sum_{t_k\in D_n, t_k<t} g(Y_{t_{k}}^\theta)\int_{t_{k-1}}^{t_k} b^y(T-s,Y^\theta_s,B_s)\, \di s  
\\& = \sum_{t_k\in D_n, t_k<t} g(\bar Y^\theta_{T-t_{k}})\int_{T-t_{k-1}}^{T-t_k} b^y(s,\bar Y^\theta_s,\bar B_s)\, \di s \\
& \to \int_{T-t}^T g(\bar{Y}^\theta_{T-s})b^y(s,\bar Y_s^\theta,\bar B_s)\, \di s 
= \int_0^t g(Y_s^\theta) b^y(T-s,Y_s^\theta,B_s)\,\di s,\ n\to\infty,
\end{aligned}
\end{equation}
uniformly on $[0,T]$ in probability.

We will first establish an estimate for $\int_0^T |g(Y_s^\theta)  b^y(T-s,Y_s^\theta,B_s)|\, \di s$
 with $g\in L^2(\bR)$. 
From \eqref{eq:b-y}, using the Cauchy--Schwarz inequality, 
we have 
\begin{align*}
& \int_0^T\E\Big|g(Y_s^\theta) b^y(s,Y_s^\theta,B_s)\Big|\, \di s 
\\& \leq C\int_0^T\E |g(Y_s^\theta) |\Big( \frac{1}{2Y_s^\theta\beta^2(Y^\theta_s) - B_s} + \frac{2Y^\theta_s\beta^2(Y^\theta_s) - B_s}{s}\Big)\, \di s \\
& \leq C \int_0^T \Big[\E |g(Y^\theta_s)|^2\cdot  \E \Big( \frac{1}{\big(2Y_s^\theta\beta^2(Y_s^\theta) - B_s\big)^2}
+  \frac{\big(2Y_s^\theta\beta^2(Y_s^\theta) - B_s\big)^2}{s^2}\Big) \Big]^{\frac12}\,\di s.
\end{align*}
As before, from the estimate $p(s,Y^\theta_s)\le \frac{C}{\sqrt{s}}$ it follows that
$$
\E|g(Y_s^\theta)|^2 \le \frac{C}{\sqrt{s}}\|g\|_{L^2(\bR)}^2.
$$
Further, using \eqref{2Y_t - B_t} and \eqref{eq:BLdensity} with $w_0 = 0$, we get
\ban
 \E  \frac{1}{\big(2Y^\theta_s\beta^2(Y_s^\theta) - B_s\big)^2}  & = \E  \frac{1}{\theta^2\big(B_s^\theta +L_s(B^\theta)\big)^2} \\
& = \frac{1}{\sqrt{2\pi s^3}}\int_0^\infty \int_{-\infty}^\infty\frac{\beta(b)}{l + |b|} \ex^{-\frac{(l+|b|)^2}{2s}}\,\di b\, \di l \\
& \le \frac{C}{\sqrt{s^3}} \int_{0}^\infty\int_{0}^\infty\frac{1}{l + b} \ex^{-\frac{(l+b)^2}{2s}}\,\di b\, \di l
   \le  \frac{C}{\sqrt{s^3}} \int_0^{\infty} \ex^{-\frac{z^2}{2s}}\, \di z \le \frac{C}{s}. 
\ean
Similarly, 
\ban
& \E \big(2Y_s^\theta\beta^2(Y^\theta_s) - B_s\big)^2 
\le \frac{C}{\sqrt{s^3}} \int_{0}^\infty\int_{0}^\infty (l+b)^3 \ex^{-\frac{(l+b)^2}{2s}}\, \di l\,\di b
\le  
\frac{C}{\sqrt{s^3}} \int_{0}^\infty z^4 \ex^{-\frac{z^2}{2s}}\di z \le Cs. 
\ean
Therefore, 
\begin{equation}
\label{eq:inthby}
\begin{aligned}
&\int_0^T\E |g(Y_s^\theta)b^y(s,Y_s^\theta,B_s) |\, \di s 
\\&\le C\|g\|_{L^2(\bR)}\int_{0}^T s^{-\frac{3}{4}}\, \di s \le C\|g\|_{L^2(\bR)}.
\end{aligned}
\end{equation}
Using similar estimates, we get
\begin{equation}
\label{eq:sumhby}
\begin{aligned}
 \sum_{t_k\in D_n, t_k<t}& \int_{t_{k-1}}^{t_k} \E |g(Y_{t_{k}}^\theta) b^y(T-s,Y_s^\theta,B_s)|\, \di s\\
& \le \sum_{t_k\in D_n, t_k<t} \int_{t_{k-1}}^{t_k} \Big(\E|g(Y^\theta_{t_{k}})|^2\cdot 
\E|b^y(T-s,Y^\theta_s,B_s)|^2 \Big)^{\frac12}\, \di s \\
& \leq C\|g\|_{L^2(\bR)}\sum_{t_k\in D_n, t_k<t} \int_{t_{k-1}}^{t_k} t_{k}^{-1/4} s^{-1/2}\,\di s\\
& \le C\|g\|_{L^2(\bR)}\int_{0}^{T}  s^{-3/4}\, \di s \le C\|g\|_{L^2(\bR)}.
\end{aligned}
\end{equation}
If $h\in C(\bR)$, then 
\ban
\delta_n = \max_{t_k\in D_n}\sup_{s\in[t_{k-1},t_k]} |h(Y^\theta_{t_{k}}) - h(Y^\theta_s)  |\to 0,\ n\to\infty,
\ean
almost surely, and we can estimate
\begin{align*}
& \Big|\sum_{t_k\in D_n, t_k<t} h(Y^\theta_{t_{k}})\int_{t_{k-1}}^{t_k} b^y(T-s,Y^\theta_s,B_s)\, \di s 
- \int_0^t h(Y^\theta_s)b^y(T-s,Y^\theta_s,B_s)\,\di s\Big|
\\
&\quad \le \delta_n \int_0^T |b^y(T-s,Y_s^\theta,B_s)|\,\di s.
\end{align*}
Similarly to the calculations above, 
\ban
\int_0^T\E |b^y(T-s,Y_s^\theta,B_s)|\, \di s  
\le \int_0^T\Big(\E|b^y(T-s,Y_s^\theta,B_s)|^2\Big)^{\frac12}\,\di s\le C\int_0^T s^{-\frac12}\,\di s \le C,
\ean
so that $\int_0^T |b^y(T-s,Y_s^\theta,B_s)|\, \di s$ is bounded in probability. Therefore, for $h\in C(\bR)$,
\ban
\sum_{t_k\in D_n, t_k<t} h(Y^\theta_{t_{k}})\int_{t_{k-1}}^{t_k} b^y(T-s,Y_s^\theta,B_s)\, \di s 
\to \int_0^t h(Y_s^\theta)b^y(T-s,Y^\theta_s,B_s)\,\di s,\ n\to\infty,
\ean
uniformly on $[0,T]$ in probability. Hence, taking, as before, 
a sequence $h_m\in C(\bR)$ converging to $g$ in $L^2(\bR)$ and using \eqref{eq:inthby} and \eqref{eq:sumhby}, we arrive at 
\eqref{eq:sumgb->intgb}.
Combined with our previous findings, this leads to
\ban
&\sum_{t_k\in D_n,t_k<t} \Big(f(B^\theta_{t_k})-f(B^\theta_{t_{k-1}})\Big)(B_{t_k}-B_{t_{k-1}}) \to \int_0^t f(B^\theta_{s})\,\di B_s - \int_0^t f(B^\theta_{s})\,\di^* B_s, \ n\to\infty,
\ean
uniformly on $[0,T]$ in probability. Since $T>0$ is arbitrary, this  means precisely that the desired u.c.p.\ convergence holds. 
\end{proof}

\section{Proof of Theorem \ref{t:main}} \label{sec:nine}

For definiteness we set $X_0=0$. Let us first show that $X^\theta$ is a strong solution to \eqref{e:1}. For $\theta = 0$ and $\alpha\in (-1,1)$, the statement follows directly from the It\^o formula proven in \cite*{FPSh-95}. 
Let $\theta\in [-1,1]\setminus\{0\}$ and $\alpha\in (0,1)$. 
If $h\in C^1(\bR)$, $h(0) = 0$, $H(x) = \int_0^x h(y)\, \di y$, then by the usual It\^o formula for semimartingales (see, e.g.\ 
\cite[Theorem II.32]{Protter-04}), we have 
\ban
H(B_t^\theta) 
&= H(0) + \int_0^t h(B_s^\theta)\,\di B_s +\theta \int_0^t h(B_s^\theta)\,\di L(B^\theta) 
+ \frac12[h(B^\theta),B]_t+\frac{\theta}{2} [h(B^\theta),L(B^\theta)]_t,
\ean
where the decomposition of quadratic variation into the sum holds true
since both $[h(B^\theta),B^\theta]$ and $[h(B^\theta),B]$ exist as u.c.p.\ limits. 
Furthermore since $h(0)=0$, the quadratic variation $[h(B^\theta),L(B^\theta)]$ and the integral w.r.t.\ $L(B^\theta)$ vanish a.s., so that
we obtain the equality
\ban
H(B_t^\theta)& = H( 0 ) + \int_0^t h(B_s^\theta)\,\di B_s+ \frac12[h(B^\theta),B]_t.
\ean
Taking a sequence $\{h_m\}$ of $C^1$-functions such that, $h_m(0)=0$, $h_m(x)=|(1-\alpha)x|^\alpha$ for $|x|\geq 1$ and 
$\sup_{x\in[0,1]}|h_m(x)-(1-\alpha)|x|^\alpha|\to 0$, $m\to\infty$, we utilize the
It\^o isometry and Theorem \ref{t:cov} to get the desired result.

Concerning the uniqueness, by Theorem~\ref{t:ws}, any strong solution must be given by \eqref{e:xtheta} with some $\theta\in[-1,1]$. 
So it remains to show that for $\theta\neq 0$ and $\alpha\in (-1,0]$, $X^\theta$ is not a solution of the SDE.

Let $\alpha=0$.
Clearly,
\ban
\int_0^t \bI(B_s^\theta\neq 0)\,\di B_s=B_t\quad \text{a.s.}
\ean
However
\ban
{}[\bI(B_\cdot^\theta\neq 0),B]\equiv 0 \quad \text{a.s.}
\ean
since $h(x)=\bI(x\neq 0)$ can be approximated by $h_m(x)\equiv 1$ in $L^2(\bR)$, and $[1,B]\equiv 0$.
Hence,
\ban
\int_0^t \bI(B_s^\theta\neq 0)\circ\di B_s=B_t\neq X^\theta_t=B_t+\theta L_t(B^\theta).
\ean
\smallskip
\noindent

For $\alpha\in (-1,0)$, the Stratonovich integral w.r.t.\ $B$ is well defined as the sum
\ban
\int_0^t |B_s^\theta|^\frac{\alpha}{1-\alpha}\circ \di B_s= \int_0^t |B_s^\theta|^\frac{\alpha}{1-\alpha}\,\di B_s
+\frac12 [  |B^\theta|^\frac{\alpha}{1-\alpha} ,B]_t
\ean
Choosing again a sequence of $C^1$-functions $\{h_n\}$ such that 
\ban
&h_m(x)\equiv |(1-\alpha)x|^\frac{\alpha}{1-\alpha},\quad |x|\geq 1,\\ 
&\|h_m(\cdot) - |(1-\alpha)(\, \cdot\, )|^\frac{\alpha}{1-\alpha}\|_{L^2(\bR)}\to 0,
\ean
we obtain that
\ban
&H_m(x)=\int_0^x h_m(y)\,\di y\to H(x)=\big( (1-\alpha) x\big)^\frac{1}{1-\alpha}
\ean
uniformly on $\bR$, so that we can apply the standard It\^o formula to obtain
\ban
H_m(B_t^\theta)& = H_m( 0 ) + \int_0^t h_m(B_s^\theta)\,\di B^\theta_s+ \frac12[h_m(B^\theta),B^\theta]_t\\
&=H_m( 0 ) + \int_0^t h_m(B_s^\theta)\circ \di B_s+\theta \int_0^t h_m(B^\theta_s)\circ\di L_s(B^\theta).
\ean
Passing to the limit as $m\to\infty$ we observe that $H_m(B^\theta_t)\to H(B^\theta_t)=X^\theta_t$ as well as 
\ban
\int_0^t h_m(B_s^\theta)\circ \di B_s&=\int_0^t h_m(B_s^\theta)\, \di B_s+\frac12 [h_m(B^\theta),B]_t\\
&\to 
\int_0^t h(B_s^\theta)\, \di B_s+\frac12 [h(B^\theta),B]_t=\int_0^t h(B_s^\theta)\circ \di B_s=\int_0^t |X_s^\theta|^\alpha\circ \di B_s
\ean
by the It\^o isometry and Theorem \ref{t:cov}.
However it is easy to see e.g.\ by the monotone convergence (if we choose $h_n$ monotonically increasing) that
\ban
&\int_0^t h(B^\theta_s)\circ \di L_s 
= \lim_{m\to\infty}\int_0^t h_m(B^\theta_s)\circ \di L_s\\ 
&= (1-\alpha)^\frac{\alpha}{1-\alpha} \lim_{m\to\infty} \lim_{n\to\infty}\sum_{t_k< t}
\frac{h_m(B^\theta_{t_k})+h_m(B^\theta_{t_{k-1}})}{2}(L_{t_k}-L_{t_{k-1}})=+\infty,
%
\ean
so that the SDE \eqref{e:1} is not satisfied unless $\theta=0$.

Note that the Riemann-Stieltjes integral w.r.t.\ $L$ does not exist since the points of increase of $L$ coincide with the points of 
discontinuity of $|B^\theta|^{\frac{\alpha}{1-\alpha}}$.


\begin{thebibliography}{32}
\providecommand{\natexlab}[1]{#1}
\providecommand{\url}[1]{\texttt{#1}}
\expandafter\ifx\csname urlstyle\endcsname\relax
  \providecommand{\doi}[1]{doi: #1}\else
  \providecommand{\doi}{doi: \begingroup \urlstyle{rm}\Url}\fi

\bibitem[Appuhamillage et~al.(2011)Appuhamillage, Bokil, Thomann, Waymire, and
  Wood]{Appuhamillage2011corrections}
T.~Appuhamillage, V.~Bokil, E.~Thomann, E.~Waymire, and B.~Wood.
\newblock {Corrections for ``Occupation and local times for skew Brownian
  motion with applications to dispersion across an interface''}.
\newblock \emph{The Annals of Applied Probability}, 21\penalty0 (5):\penalty0
  2050--2051, 2011.

\bibitem[Aryasova and Pilipenko(2011)]{AryasovaP-11}
O.~V. Aryasova and A.~Yu. Pilipenko.
\newblock {On the strong uniqueness of a solution to singular stochastic
  differential equations}.
\newblock \emph{Theory of Stochastic Processes}, 17(33)\penalty0 (2):\penalty0
  1--15, 2011.

\bibitem[Barlow et~al.(2000)Barlow, Burdzy, Kaspi, and Mandelbaum]{BBKM-00}
M.~Barlow, K.~Burdzy, H.~Kaspi, and A.~Mandelbaum.
\newblock {Variably skewed Brownian motion}.
\newblock \emph{Electronic Communications in Probability}, 5:\penalty0 57--66,
  2000.

\bibitem[Bass et~al.(2007)Bass, Burdzy, and Chen]{BassBC-07}
R.~F. Bass, K.~Burdzy, and Z.-Q. Chen.
\newblock {Pathwise uniqueness for a degenerate stochastic differential
  equation}.
\newblock \emph{The Annals of Probability}, page 2385--2418, 2007.

\bibitem[Beck(1973)]{Beck-73}
A.~Beck.
\newblock {Uniqueness of flow solutions of differential equations}.
\newblock In \emph{{Recent Advances in Topological Dynamics}}, volume 318 of
  \emph{{Lecture Notes in Mathematics}}, page 36--56. Springer, Berlin, 1973.

\bibitem[Cherny(2001)]{Cherny-01}
A.~S. Cherny.
\newblock {Principal values of the integral functionals of Brownian motion:
  Existence continuity and an extension of It?'s formula}.
\newblock In \emph{{S?minaire de Probabilit?s XXXV}}, volume 1755 of
  \emph{{Lecture Notes in Mathematics}}, page 348--370. Springer, Berlin,
  2001.

\bibitem[Cherny and Engelbert(2005)]{CheEng05}
A.~S. Cherny and H.-J. Engelbert.
\newblock \emph{{Singular stochastic differential equations}}, volume 1858 of
  \emph{{Lecture Notes in Mathematics}}.
\newblock Springer, 2005.

\bibitem[Cherstvy et~al.(2013)Cherstvy, Chechkin, and Metzler]{CherstvyCM-13}
A.~G. Cherstvy, A.~V. Chechkin, and R.~Metzler.
\newblock {Anomalous diffusion and ergodicity breaking in heterogeneous
  diffusion processes}.
\newblock \emph{New Journal of Physics}, 15:\penalty0 083039, 2013.

\bibitem[Denisov and Horsthemke(2002)]{DenHor-02}
S.~I. Denisov and W.~Horsthemke.
\newblock {Statistical properties of a class of nonlinear systems driven by
  colored multiplicative Gaussian noise}.
\newblock \emph{Physical Review E}, 65:\penalty0 031105, 2002.

\bibitem[Dynkin(1965)]{Dynkin65}
E.~B. Dynkin.
\newblock \emph{{Markov Processes}}.
\newblock Springer, 1965.

\bibitem[Engelbert and Schmidt(1981)]{EngelbertSchmidt-81}
H.-J. Engelbert and W.~Schmidt.
\newblock {On the behaviour of certain functionals of the Wiener process and
  applications to stochastic differential equations}.
\newblock In \emph{{Stochastic Differential Systems}}, volume~36 of
  \emph{{Lecture Notes in Control and Information Sciences}}, page 47--55,
  Berlin, 1981. Springer.

\bibitem[Engelbert and Schmidt(1985)]{EngelbertSchmidt-85}
H.-J. Engelbert and W.~Schmidt.
\newblock {On solutions of one-dimensional stochastic differential equations
  without drift}.
\newblock \emph{Zeitschrift f?r Wahrscheinlichkeitstheorie und verwandte
  Gebiete}, 68\penalty0 (3):\penalty0 287--314, 1985.

\bibitem[{\'E}tor{\'e} and Martinez(2012)]{etore2012existence}
P.~{\'E}tor{\'e} and M.~Martinez.
\newblock On the existence of a time inhomogeneous skew brownian motion and
  some related laws.
\newblock \emph{Electronic Journal of Probability}, 17\penalty0 (19):\penalty0
  1--27, 2012.

\bibitem[F{\"o}llmer et~al.(1995)F{\"o}llmer, Protter, and Shiryaev]{FPSh-95}
H.~F{\"o}llmer, P.~Protter, and A.~N. Shiryaev.
\newblock {Quadratic covariation and an extension of It\^o's formula}.
\newblock \emph{Bernoulli}, 1\penalty0 (1/2):\penalty0 149--169, 1995.

\bibitem[Gairat and Shcherbakov(2017)]{gairat2017density}
A.~Gairat and V.~Shcherbakov.
\newblock {Density of Skew Brownian motion and its functionals with application
  in finance}.
\newblock \emph{Mathematical Finance}, 17\penalty0 (4):\penalty0 1069--1088,
  2017.

\bibitem[Girsanov(1962)]{Girsanov-62}
I.~V. Girsanov.
\newblock {An example of non-uniqueness of the solution of the stochastic
  equation of K. Ito}.
\newblock \emph{Theory of Probability \& Its Applications}, 7\penalty0
  (3):\penalty0 325--331, 1962.

\bibitem[Harrison and Shepp(1981)]{HShepp-81}
J.~M. Harrison and L.~A. Shepp.
\newblock {On skew Brownian motion}.
\newblock \emph{The Annals of Probability}, 9\penalty0 (2):\penalty0 309--313,
  1981.

\bibitem[Haussmann and Pardoux(1985)]{haussmann1985time}
U.~G. Haussmann and E.~Pardoux.
\newblock {Time reversal of diffusion processes}.
\newblock In Metivier M. and Pardoux E., editors, \emph{{Stochastic
  Differential Systems Filtering and Control}}, volume~69 of \emph{{Lecture
  Notes in Control and Information Sciences}}, page 176--182. Springer,
  Berlin, 1985.

\bibitem[Haussmann and Pardoux(1986)]{haussmann1986time}
U.~G. Haussmann and E.~Pardoux.
\newblock {Time reversal of diffusions}.
\newblock \emph{The Annals of Probability}, 14\penalty0 (4):\penalty0
  1188--1205, 1986.

\bibitem[Keilson and Wellner(1978)]{KeilsonW-78}
J.~Keilson and J.~A. Wellner.
\newblock {Oscillating Brownian motion}.
\newblock \emph{Journal of Applied Probability}, 15\penalty0 (2):\penalty0
  300--310, 1978.

\bibitem[Lejay(2006)]{Lejay-06}
A.~Lejay.
\newblock {On the constructions of the skew Brownian motion}.
\newblock \emph{Probability Surveys}, 3:\penalty0 413--466, 2006.

\bibitem[Lejay and Pigato(2018)]{LejPig18}
A.~Lejay and P.~Pigato.
\newblock {Statistical estimation of the Oscillating Brownian motion}.
\newblock \emph{Bernoulli}, 24\penalty0 (4B):\penalty0 3568--3602, 2018.

\bibitem[Mao(2007)]{Mao2007}
X.~Mao.
\newblock \emph{{Stochastic Differential Equations and Applications}}.
\newblock Woodhead Publishing, Cambridge, second edition, 2007.

\bibitem[McKean(1969)]{McKean-1969}
H.~P. McKean.
\newblock \emph{{Stochastic integrals}}.
\newblock Academic Press, New York, 1969.

\bibitem[Petit(1997)]{Petit-97}
F.~Petit.
\newblock {Time reversal and reflected diffusions}.
\newblock \emph{Stochastic Processes and their Applications}, 69:\penalty0
  25--53, 1997.

\bibitem[Protter(2004)]{Protter-04}
P.~E. Protter.
\newblock \emph{{Stochastic Integration and Differential Equations}}, volume~21
  of \emph{{Applications of Mathematics}}.
\newblock Springer, Berlin, second edition, 2004.

\bibitem[Russo and Vallois(1993)]{russo1993forward}
F.~Russo and P.~Vallois.
\newblock {Forward, backward and symmetric stochastic integration}.
\newblock \emph{Probability Theory and Related Fields}, 97\penalty0
  (3):\penalty0 403--421, 1993.

\bibitem[Russo and Vallois(1995)]{russo1995generalized}
F.~Russo and P.~Vallois.
\newblock {The generalized covariation process and It? formula}.
\newblock \emph{Stochastic Processes and their Applications}, 59\penalty0
  (1):\penalty0 81--104, 1995.

\bibitem[Russo and Vallois(2000)]{russo2000stochastic}
F.~Russo and P.~Vallois.
\newblock {Stochastic calculus with respect to continuous finite quadratic
  variation processes}.
\newblock \emph{Stochastics}, 70\penalty0 (1--2):\penalty0 1--40, 2000.

\bibitem[Varadhan(2011)]{Varadhan-RBM}
S.~R.~S. Varadhan.
\newblock \emph{{Chapter 16. Reflected Brownian Motion}}.
\newblock 2011.
\newblock \verb|https://math.nyu.edu/~varadhan/Spring11/topics16.pdf|.

\bibitem[Weinryb(1983)]{weinryb1983etude}
S.~Weinryb.
\newblock Etude d'une equation diff{\'e}rentielle stochastique avec temps
  local.
\newblock \emph{S{\'e}minaire de probabilit{\'e}s (Strasbourg)}, 17:\penalty0
  72--77, 1983.

\bibitem[Zvonkin(1974)]{zvonkin74}
A.~K. Zvonkin.
\newblock {A transformation of the phase space of a diffusion process that
  removes the drift}.
\newblock \emph{Mathematics of the USSR--Sbornik}, 22\penalty0 (1):\penalty0
  129, 1974.

\end{thebibliography}

\appendix

\section{Proof of Lemma \ref{l:ito}}\label{a:ito}

To avoid cumbersome notation, we will use $X$ and $B$ in place of $\widetilde{X}$ and $\widetilde{B}$. 

Let $\e>0$ be such that $g(x)= 0$ for $|x|\leq 2\e$.
Without loss of generality we may assume that $g$, $g'$ and $g''$ are bounded on $\bR$ and 
$\phi$, $\phi'$ 
and $\phi''$ are bounded for $|x|\geq \e$ (otherwise we perform localization and consider the solution $X$ stopped 
after hitting an arbitrary threshold $\pm R$, $R>0$).

For $n\geq 1$, let $D_n = \{0=t_0^n<t_1^n<\cdots<t_n^n=t\}$ be a sequence of partitions of $[0,t]$ with the mesh
$\|D_n\|=\max_k|t^n_k-t^n_{k-1}|\to 0$, $n\to\infty$ (in the following we may omit the index $n$ for brevity).
 
First note that due to the additivity of the Stratonovich integral, for any $r\geq 0$ we have 
\ban
X_t=X_r+ \int_r^t \phi(X_s)\,\di B_s + \frac12[\phi(X),B]_{r,t}
\ean
and defining 
\ba
\tau_r^\e=\inf \{t\geq r\colon |X_t|\leq \e\}
\ea
we get that 
\ban
X_{t\wedge \tau_r^\e}=X_r+ \int_r^{t\wedge \tau_r^\e} \phi(X_s)\,\di B_s + \frac12\int_r^{t\wedge \tau_r^\e} \phi(X_s)\phi'(X_s)\,\di s
\ean
and in particular, for $\omega\in\{\tau_r^\e\leq t\}$
\ba
&{}[\phi(X),B]_{r,s}=\int_r^s \phi(X_u)\phi'(X_u)\,\di u,\\
&\Big |[\phi(X),B]_{r,s}\Big|\leq \max_{|x|\geq \e}|\phi(x)\phi'(x)| \cdot |s-r|,\quad r\leq s\leq t.
\ea
Let $w(\delta)$ be the modulus of continuity of $X$ on the interval $[0,t]$, namely
\ba
w(\delta)=w(\omega,\delta)=\sup_{\substack{|u-v|\leq \delta\\u,v\in[0,t] }}|X_u(\omega)-X_v(\omega)|.
\ea
Obviously $\lim_{\delta\to 0}w(\delta)=0$ a.s.
Denote by
\ba
w_n:=w(\|D_n\|)
\ea
and let $A_n^\e:=\{\omega\colon w_n<\e\}$. Then $\bI_{A_n^\e}(\omega)\to 1$ a.s.\ as $n\to \infty$.
Moreover, for $\omega\in A_n^\e$ we have:
\ba
\text{if}\quad   \sup_{s\in [t_k^n,t^n_{k+1}]}|X_{s}(\omega)|\geq 2\e \quad \text{then} 
\quad \inf_{s\in [t_k^n,t^n_{k+1}]}|X_{s}(\omega)|\geq \e .
\ea
Let us also denote
\ba
\bI_{t_k,s}^\e:=\bI\Big(\inf_{u\in[t_k,s]}|X_u|\geq \e\Big)
\ea
and note that for $\omega\in A_n^\e$ we have
\ba
\text{if $|X_{t_k}(\omega)|\geq 2\e$} \quad \text{then}\quad 
\bI_{t_k,s}^\e(\omega)\equiv 1,\ s\in[t_k,t_{k+1}].
\ea
Let us write $g(X_t)$ as a telescopic sum 
\ba
g(X_t)&=g(X_0)+\sum_{k}\Big(g(X_{t_{k+1}})-g(X_{t_k})\Big)\\
&=g(X_0)+\sum_k g'(X_{t_{k}}) (X_{t_{k+1}}-X_{t_k})+\frac12 g''(\xi_k) (X_{t_{k+1}}-X_{t_k})^2\\
&=g(X_0)+\sum_k g'(X_{t_{k}}) \int_{t_k}^{t_{k+1}}\phi(X_s)\,\di B_s\\
&\qquad \qquad +\frac12\sum_k g'(X_{t_{k}}) \cdot  [\phi(X),B]_{{t_k},{t_{k+1}}}\\
&\qquad \qquad +\frac12 \sum_k g''(\xi_k) (X_{t_{k+1}}-X_{t_k})^2,\quad \xi_k\in (X_{t_k},X_{t_{k+1}}).
\ea 
We will prove that the first sum converges to $\int_0^t g'(X_s)\phi(X_s)\,\di s$, the second sum converges in probability
to $\frac12\int_0^t  g'(X_s)\phi'(X_s)\,\di s$ 
whereas the third sum converges to $\frac12\int_0^t  g''(X_s)^2\phi(X_s)\,\di s$.

\noindent
1. Note that the It\^o integral $t\mapsto \int_0^t g'(X_{s}) \phi(X_s)\,\di B_s$ is well-defined on $[0,T]$ due to the boundedness of $x\mapsto g'(x) \phi(x)$. 
For any $\eta>0$, we estimate:
\ba
& \P\Big( \Big|\sum_k g'(X_{t_{k}}) \int_{t_k}^{t_{k+1}}\phi(X_s)\,\di B_s  - \int_0^t g'(X_{s}) \phi(X_s)\,\di B_s    \Big|\geq \eta  \Big)\\
&\leq 
\P\Big( \bI_{A_n^\e}\cdot \Big|\sum_k g'(X_{t_{k}}) \int_{t_k}^{t_{k+1}}\phi(X_s)\,\di B_s  - \int_0^t g'(X_{s}) \phi(X_s)\,\di B_s    \Big|\geq \eta  \Big)
+\P(\bar A_n^\e)\\
&\leq \eta^{-2}\E \bI_{A_n^\e}\Big[\sum_k \int_{t_k}^{t_{k+1}} \Big(g'(X_{t_{k}})- g'(X_{s})\Big)\phi(X_s)\,\di B_s    \Big]^2 +\P(\bar A_n^\e)\\
&=\eta^{-2}\E \bI_{A_n^\e}\Big[\sum_k \int_{t_k}^{t_{k+1}} \Big(g'(X_{t_{k}})- g'(X_{s})\Big)\bI\Big(\sup_{u\in[t_k,s]} |X_u-X_{t_k}|\leq \e\Big)\phi(X_s)\,\di B_s \Big]^2 
+  \P(\bar A_n^\e)\\
&\leq \eta^{-2}\E  \Big[\sum_k \int_{t_k}^{t_{k+1}} \Big(g'(X_{t_{k}})- g'(X_{s})\Big)\bI\Big(\sup_{u\in[t_k,s]} |X_u-X_{t_k}|\leq \e\Big)\phi(X_s)\,\di B_s \Big]^2 
+ \P(\bar A_n^\e)\\
&\leq \eta^{-2} \E \int_0^t  \Big|\sum_k \Big(g'(X_{t_{k}})- g'(X_{s})\Big)\bI\Big(\sup_{u\in[t_k,s]} |X_u-X_{t_k}|\leq \e\Big)\phi(X_s)
\bI_{(t_k,t_{k+1}]}(s) \Big|^2\,\di s 
+  \P(\bar A_n^\e)\to 0
\ea
where we have used 
the continuity of $X$, boundedness of $g'$ and $\phi$, and the fact that with probability 1 for each $s\in[0,t]$
\ba
\sum_k \Big(g'(X_{t_{k}})- g'(X_{s})\Big)\bI\Big(\sup_{u\in[t_k,s]} |X_u-X_{t_k}|\leq \e\Big)\phi(X_s)
\bI_{[t_k,t_{k+1})}(s)\to 0,\quad n \to\infty.
\ea

\noindent 
2. 
For any $\eta>0$, we estimate:
\ba
&\P\Big(\Big|  \sum_{k}g'(X_{t_{k}})\cdot   [\phi(X),B]_{{t_k},{t_{k+1}}}-\int_0^t g'(X_s)\phi(X_s)\phi'(X_s)\,\di s    \Big|\geq \eta \Big)\\
&\leq \P\Big(\bI_{A_n^\e}\Big|  \sum_{k}g'(X_{t_{k}}) \cdot   [\phi(X),B]_{{t_k},{t_{k+1}}}-\int_0^t g'(X_s)\phi(X_s)\phi'(X_s)\,\di s    \Big|\geq \eta \Big)
+\P(\bar A_n^\e)\\
&= \P\Big(\bI_{A_n^\e}\Big|  \sum_{k}g'(X_{t_{k}}) \cdot   \int_{t_k}^{t_{k+1}} \phi(X_s)\phi'(X_s)\,\di s  
-\int_0^t g'(X_s)\phi(X_s)\phi'(X_s)\,\di s  \Big|\geq \eta \Big)
+\P(\bar A_n^\e)\\
&=\P\Big(\bI_{A_n^\e}\Big| \sum_{k} \int_0^t  \Big( g'(X_{t_{k}})- g'(X_s)\Big)\bI^\e_{t_k,s} \cdot \phi(X_s)\phi'(X_s)\,\di s \Big|\geq \eta \Big)
+\P(\bar A_n^\e)\to 0,\quad n\to\infty.
\ea

\noindent 
3. Eventually, to establish the convergence of the third summand we write the following decomposition
\ba
\sum_k g''(\xi_k) & (X_{t_{k+1}}-X_{t_k})^2-\int_0^t g''(X_s)\phi^2(X_s)\,\di s\\
&=\sum_k g''(X_{t_k})\Big[ \Big(\int_{t_k}^{t_{k+1}} \phi(X_s)\,\di B_s\Big)^2 -\int_{t_k}^{t_{k+1}}  \phi^2(X_s)\,\di s\Big] \\
&+\sum_k \int_{t_k}^{t_{k+1}} (g''(X_{t_k})- g''(X_s))\phi^2(X_s)\,\di s \\
&+\sum_k \Big(g''(\xi_k)-g''(X_{t_k})\Big) \Big(\int_{t_k}^{t_{k+1}} \phi(X_s)\,\di B_s\Big)^2  \\
&+2\sum_k g''(\xi_k) \int_{t_k}^{t_{k+1}} \phi(X_s)\,\di s\cdot [\phi(X),B]_{t_k,t_{k+1}} \\
&+\sum_k g''(\xi_k) [\phi(X),B]_{t_k,t_{k+1}}^2 \\
&=S_1+S_2+S_3+S_4+S_5
\ea
and show that each of $S_i$, $i=1,\dots, 5$, converges to zero in probability.

\noindent 
3.1 For any $\eta\geq 0$,
\ba
\P(|S_1| &\geq \eta)\leq \eta^{-2}\E \bI_{A_n^\e} S_1^2+ \P(\bar A_n^\e)\\
&=\E \Big[ \bI_A \sum_k g''(X_{t_k})\cdot 
\Big(\Big(\int_{t_k}^{t_{k+1}}\bI^\e_{t_k,s}\cdot \phi(X_s)\,\di B_s\Big)^2 -\int_{t_k}^{t_{k+1}} \bI^\e_{t_k,s}\cdot  \phi^2(X_s)\,\di s\Big)\Big]^2
+ \P(\bar A_n^\e)\\
&\leq \E \Big[\sum_k g''(X_{t_k})\cdot 
\Big(\Big(\int_{t_k}^{t_{k+1}}\bI^\e_{t_k,s}\cdot \phi(X_s)\,\di B_s\Big)^2 -\int_{t_k}^{t_{k+1}} \bI^\e_{t_k,s}\cdot  \phi^2(X_s)\,\di s\Big)\Big]^2
+ \P(\bar A_n^\e)\\
&= \E |S_{11}|^2+ \P(\bar A_n^\e).
\ea
Denoting
\ba
M_{k+1}:=g''(X_{t_k})\cdot 
\Big(\Big(\int_{t_k}^{t_{k+1}}\bI^\e_{t_k,s}\cdot \phi(X_s)\,\di B_s\Big)^2 -\int_{t_k}^{t_{k+1}} \bI^\e_{t_k,s}\cdot  \phi^2(X_s)\,\di s\Big)
\ea
we observe that the process $(M_{k+1},\rF_{t_k})_{k=0,\dots,n-1}$ is a zero mean square integrable martingale, and that
\ba
\E |S_{11}|^2&= \E \sum_{k=0}^{n-1} M_{k+1}^2+2\E\sum_{i<j}M_{i+1}M_{j+1}.
\ea
Then clearly the second sum equals to zero (to see this note that 
$\E[M_{i+1}M_{j+1}]=\E [\E [M_{i+1}M_{j+1}|\rF_{t_j}]]=\E [M_{i+1}\E [M_{j+1}|\rF_{t_j}]]=0$). For the first sum we get
\ba
 \E \sum_k M_{k+1}^2 &\leq 2 \E \sum_k  |g''(X_{t_k})|^2 \Big(\int_{t_k}^{t_{k+1}}\bI^\e_{t_k,s}\cdot  \phi(X_s)\,\di B_s\Big)^4
\\&\qquad +2 \E\sum_k |g''(X_{t_k})|^2 \Big(\int_{t_k}^{t_{k+1}}\bI^\e_{t_k,s}\cdot   \phi^2(X_s)\,\di s\Big)^2\\
&\leq 2\cdot 36\cdot \|g''\|^2\cdot\sum_k (t_{k+1}-t_k) \E \int_{t_k}^{t_{k+1}}\bI^\e_{t_k,s}\cdot  |\phi(X_s)|^4\,\di s\\
&+2 \|g''\|^2\cdot \sup_{|x|\geq \e} |\phi(x)|^4\cdot\sum_k (t_{k+1}-t_k)^2\\
&\leq 74 t\cdot \|g''\|^2\cdot \sup_{|x|\geq \e}|\phi(x)|^4\cdot \|\Delta t^n\|\to 0, n\to \infty,
\ea
where in the second line we have applied the moment inequality for It\^o integrals, see, e.g.\ Theorem I.7.1 in 
\cite{Mao2007}.

\noindent 
3.2 To estimate $S_2$, we note that
\ba
S_2&=\Big|\sum_k \int_{t_k}^{t_{k+1}} (g''(X_{t_k})- g''(X_s))\phi^2(X_s)\,\di s\Big|\\&\leq
\max_k \sup_{s\in[t_k,t_{k+1}]}|g''(X_{t_k})- g''(X_s)| \cdot \int_0^t \phi^2(X_s)\,\di s\to 0\quad \text{a.s.}
\ea

\noindent 
3.3
Analogously, we have
\ba
S_3&=\Big|\sum_k \Big(g''(\xi_k)-g''(X_{t_k})\Big) \Big(\int_{t_k}^{t_{k+1}} \phi(X_s)\,\di B_s\Big)^2|\\
&\leq \max_k \sup_{s\in[t_k,t_{k+1}]}|g''(X_{t_k})- g''(X_s)| \cdot  \sum_k \Big(\int_{t_k}^{t_{k+1}} \phi(X_s)\,\di B_s\Big)^2\to 0\quad \text{a.s.}
\ea
since the sums in the previous line are bounded in probability:
\ba
\E \sum_k \Big(\int_{t_k}^{t_{k+1}} \phi(X_s)\,\di B_s\Big)^2=\E \int_0^t \phi^2(X_s)\,\di s<\infty.
\ea

\noindent 
3.4
For any $\eta>0$ we get for $S_4$:
\ba
& \P(|S_4|\geq \eta) 
\leq\P\Big(\bI_{A_n^\e}\sum_k |g''(\xi_k)| \int_{t_k}^{t_{k+1}} |\phi(X_s)|\,\di s\cdot \Big|[\phi(X),B]_{t_k,t_{k+1}}\Big|\geq \eta\Big)+
\P(\bar A_n^\e)\\
&\leq \P\Big(\bI_{A_n^\e}\sum_k |g''(\xi_k)| \int_{t_k}^{t_{k+1}} \bI_{t_k,s}|\phi(X_s)|\,\di s
\cdot \int_{t_k}^{t_{k+1}} \bI_{t_k,s} |\phi(X_s)\phi'(X_s)|\,\di s \geq \eta\Big)+
\P(\bar A_n^\e)\\
&\leq \P\Big( C_\e\cdot \|\Delta t^n\|\geq \eta\Big)+
\P(\bar A_n^\e),
\ea
where
\ba
C_\e=t\cdot \|g''\|\cdot \sup_{|x|\geq \e}|\phi(x)|\cdot   \sup_{|x|\geq \e}|\phi(x)\phi'(x)|.
\ea

\noindent 
3.5
Analogously, for $S_5$ we get
\ba
\P(|S_5|\geq \eta)& 
 \leq \P\Big(\bI_{A_n^\e}\sum_k |g''(\xi_k)| [\phi(X),B]_{t_k,t_{k+1}}^2\geq \eta \Big) +  \P(\bar A_n^\e)\\
&= \P\Big(\bI_{A_n^\e}\sum_k |g''(\xi_k)| \Big(\int_{t_k}^{t_{k+1}} \bI_{t_k,s} \phi(X_s)\phi'(X_s)\,\di s\Big)^2\geq \eta\Big)\\
&\leq \P\Big( C_\e\cdot \|\Delta t^n\|\geq \eta\Big)+ \P(\bar A_n^\e),
\ea
where
\ba
C_\e=t\cdot \|g''\|\cdot   \sup_{|x|\geq \e}|\phi(x)\phi'(x)|^2.
\ea

\section{Partial differential equation for $v(s,x,w)$\label{a:pde}}

\chng[R1: $Y^\theta$ defined, result announced]{
Let $\theta\in(0,1)$. Define for fixed $t>0$ and a function $g\in C^\infty(\bR^2)$ with support inside $D=\{(y,z)\in\bR^2\colon  r(y)-z>0\}$
\begin{align*}
v(s,x,w) & = \E[g(Y_t^\theta,B_t)| Y_s^\theta = x,\, B_s = w], \quad s\in [0,t],\quad  x,w\in \bR,
\end{align*}
where $Y^\theta$ is the oscillating Brownian motion introduced in \eqref{e:Y} in Section \ref{s:Y}.
We will show that this function solves equation $\left(\frac{\partial}{\partial s} + \mathcal L\right) v = 0$
with $\mathcal L$ given by \eqref{e:ell}. 
}

Denoting $\phi_t (x) = \frac{1}{\sqrt{2\pi t}} \ex^{-\frac{x^2}{2t}}$ and taking into account that 
$\psi_t(x) =  - \phi_t'(x) =  \frac{x}{t} \phi_t(x)$ we employ \eqref{eq:YBdensx} and \eqref{eq:YBdistx} to obtain
\ban
v(s,x,w)  = &  \frac{2}{\theta}\int_{-\infty}^{+\infty} \beta^2(y)
\int_{-\infty}^{r(y)-r(x)+w} g(y,z) \psi_{t-s}\Big(\frac{2y\beta^2(y) - \kappa x - z + w}{\theta}\Big)\, \di z\, \di y\\
& + \frac{\beta(x)}{\theta }\int_{y\colon xy> 0} g(y,r(y)-r(x)+w)\cdot  \Big(\phi_{t-s}(r(x)-r(y)) - \phi_{t-s}(r(x)+r(y))\Big)\,\di y.
\ean
The cases $x>0$ and $x<0$ can be treated similarly, so we will consider only the former. Denote for brevity 
\ban
\beta_+ = \frac{1+\theta}{2}, \quad 
g_1(y,z) = \frac{\partial}{\partial z}g(y,z), \quad 
g_2(y,z) = \frac{\partial^2}{\partial z^2}g(y,z), \quad 
h(x,y,z,w) = \frac{2y\beta^{2}(y) - \kappa x - z + w}{\theta}.
\ean
Also note that by \eqref{eq:l+x_0+b}, substituting $z = r(x) -r(y) + w$ into $h(x,y,z,w)$ gives $|r(x)| + |r(y)| = r(x) + |r(y)|$. 
Then 
\ban
\frac{\partial}{\partial x} v(s,x,w) & 
= - \frac{2\beta_+}{\theta} \int_{-\infty}^{\infty}\beta^2(y)\cdot g(y,r(y) - r(x) + w)\cdot \psi_{t-s}(r(x)+|r(y)|)\,\di y\\
& -\frac{2\kappa}{\theta^2} \int_{-\infty}^{\infty}\beta^2(y)\int_{-\infty}^{r(y)-r(x)+w} g(y,z)\cdot  \psi'_{t-s}(h(x,y,z,w))\,\di z\, \di y\\
& - \beta_+^2 \int_0^\infty g_1(y,r(y)-r(x)+w)\cdot \Big[\phi_{t-s}(r(x)-r(y)) - \phi_{t-s}(r(x)+r(y))\Big]\,\di y\\
& + \beta_+^2 \int_0^\infty g(y,r(y)-r(x)+w)\cdot \Big[\phi_{t-s}'(r(x)-r(y)) - \phi_{t-s}'(r(x)+r(y))\Big]\,\di y,\\
\ean
\ban
\frac{\partial}{\partial w} v(s,x,w) & 
= \frac{2}{\theta} \int_{-\infty}^{\infty}\beta^2(y)\cdot g(y,r(y) - r(x) + w)\cdot \psi_{t-s}(r(x) + |r(y)|)\,\di y \\
& +\frac{2}{\theta^2} \int_{-\infty}^{\infty}\beta^2(y)\int_{-\infty}^{r(y)-r(x)+w} g(y,z)\cdot \psi'_{t-s}(h(x,y,z,w))\,\di z\, \di y\\
& + \beta_+ \int_0^\infty g_1(y,r(y)-r(x)+w)\cdot \Big[\phi_{t-s}(r(x)-r(y)) - \phi_{t-s}(r(x)+r(y))\Big]\,\di y.
\ean
Further, 
\ban
\frac{\partial^2}{\partial x^2} v(s,x,w) & 
= \frac{2\beta_+^2}{\theta} \int_{-\infty}^{\infty}\beta^2(y)\cdot g_1(y,r(y) - r(x) + w)\cdot \psi_{t-s}(r(x) + |r(y)|)\,\di y \\
& - \frac{2\beta_+^2}{\theta} \int_{-\infty}^{\infty}\beta^2(y)\cdot g(y,r(y) - r(x) + w)\cdot \psi'_{t-s}(r(x) + |r(y)|)\,\di y\\
& +\frac{2\kappa\beta_+}{\theta^2} \int_{-\infty}^{\infty}\beta^2(y)\cdot g(y,r(y) - r(x) + w)\cdot \psi'_{t-s}(r(x) + |r(y)|)\, \di y\\
& +\frac{2\kappa^2}{\theta^3} \int_{-\infty}^{\infty}\beta^2(y)\int_{-\infty}^{r(y)-r(x)+w} g(y,z)\cdot \psi''_{t-s}(h(x,y,z,w))\,\di z\, \di y\\
& + \beta_+^3 \int_0^\infty g_2(y,r(y)-r(x)+w)\cdot\Big[\phi_{t-s}(r(x)-r(y)) - \phi_{t-s}(r(x)+r(y))\Big]\,\di y\\
& - 2\beta_+^3 \int_0^\infty g_1(y,r(y)-r(x)+w)\cdot\Big[\phi_{t-s}'(r(x)-r(y)) - \phi_{t-s}'(r(x)+r(y))\Big]\,\di y\\
& + \beta_+^3 \int_0^\infty g(y,r(y)-r(x)+w)\Big[\phi_{t-s}''(r(x)-r(y)) - \phi_{t-s}''(r(x)+r(y))\Big]\, \di y,\\
\ean
\ban
\frac{\partial^2}{\partial x\,\partial w} v(s,x,w) & = - 
\frac{2\beta_+}{\theta} \int_{-\infty}^{\infty}\beta^2(y)\cdot g_1(y,r(y) - r(x) + w)\cdot \psi_{t-s}(r(x) + |r(y)|)\,\di y \\
& -\frac{2\kappa}{\theta^2} \int_{-\infty}^{\infty}\beta^2(y)\cdot g(y,r(y) - r(x) + w)\cdot \psi'_{t-s}(r(x) + |r(y)|)\, \di y\\
& -\frac{2\kappa}{\theta^3} \int_{-\infty}^{\infty}\beta^2(y)\int_{-\infty}^{r(y)-r(x)+w} g(y,z)\cdot \psi''_{t-s}(h(x,y,z,w))\,\di z\, \di y\\
& - \beta_+^2 \int_0^\infty g_2(y,r(y)-r(x)+w)\cdot\Big[\phi_{t-s}(r(x)-r(y)) - \phi_{t-s}(r(x)+r(y))\Big]\,\di y\\
& +\beta_+^2 \int_0^\infty g_1(y,r(y)-r(x)+w)\cdot\Big[\phi_{t-s}'(r(x)-r(y)) - \phi_{t-s}'(r(x)+r(y))\Big]\,\di y,\\
\ean
\ban
\frac{\partial^2}{\partial w^2} v(s,x,w) & 
= \frac{2}{\theta} \int_{-\infty}^{\infty}\beta^2(y) g_1(y,r(y) - r(x) + w)\cdot \psi_{t-s}(r(x) + |r(y)|)\,\di y \\
& +\frac{2}{\theta^2} \int_{-\infty}^{\infty}\beta^2(y)\cdot g(y,r(y)-r(x)+w)\cdot \psi'_{t-s}(r(x) + |r(y)|)\,\di z\, \di y\\
& +\frac{2}{\theta^3} \int_{-\infty}^{\infty}\beta^2(y)\int_{-\infty}^{r(y)-r(x)+w} g(y,z)\cdot \psi''_{t-s}(h(x,y,z,w))\,\di z\, \di y\\
& + \beta_+ \int_0^\infty g_2(y,r(y)-r(x)+w)\cdot \Big[\phi_{t-s}(r(x)-r(y)) - \phi_{t-s}(r(x)+r(y))\Big]\,\di y.
\ean
Therefore, taking into account that $\frac{\kappa}{\beta_+} = \frac{1-\theta^2}{1+\theta} = 1-\theta$, we get
\ban
& \Big(\frac{1}{2\beta_+^2}\cdot \frac{\partial^2}{\partial x^2} + \frac{1}{\beta_+}\cdot \frac{\partial^2}{\partial x\,\partial w} 
+ \frac{1}{2}\cdot \frac{\partial^2}{\partial w^2} \Big) v(s,x,w)
\\
&  = \Big(-\frac{1}{\theta} + \frac{\kappa}{\beta_+ \theta^2} - \frac{2\kappa}{\beta_+ \theta^2} 
+\frac{1}{\theta^2}\Big)\int_{-\infty}^{\infty}\beta^2(y)\cdot g(y,r(y) - r(x) + w)\cdot \psi'_{t-s}(r(x) + |r(y)|)\, \di y\\
&\quad + \Big(\frac{\kappa^2}{\beta_+^2 \theta^3} -\frac{2\kappa}{\beta_+ \theta^3} 
+ \frac{1}{\theta^3}\Big)\int_{-\infty}^{\infty}\beta^2(y)\int_{-\infty}^{r(y)-r(x)+w} g(y,z) \cdot \psi''_{t-s}(h(x,y,z,w))\,\di z\, \di y\\
&\quad + \frac{\beta_+}{2} \int_0^\infty g(y,r(y)-r(x)+w)\cdot\Big[\phi_{t-s}''(r(x)-r(y)) - \phi_{t-s}''(r(x)+r(y))\Big]\,\di y 
\\
&  = \Big(-\frac{1}{\theta} - \frac{1-\theta}{\theta^2} +\frac{1}{\theta^2}\Big)
\int_{-\infty}^{\infty}\beta^2(y)\cdot g(y,r(y) - r(x) + w) \psi'_{t-s}(r(x) + |r(y)|) \, \di y\\
&\quad + \Big(\frac{(1-\theta)^2}{\theta^3} -\frac{2(1-\theta)}{\theta^3} + \frac{1}{\theta^3}\Big)
\int_{-\infty}^{\infty}\beta^2(y)\int_{-\infty}^{r(y)-r(x)+w} g(y,z)\cdot \psi''_{t-s}(h(x,y,z,w))\,\di z\, \di y\\
&\quad + \frac{\beta_+}{2} \int_0^\infty g(y,r(y)-r(x)+w)\cdot\Big[\phi_{t-s}''(r(x)-r(y)) - \phi_{t-s}''(r(x)+r(y))\Big]\,\di y
\ean
\ban
& = \frac{1}{\theta} \int_{-\infty}^{\infty}\beta^2(y)\int_{-\infty}^{r(y)-r(x)+w} g(y,z) \cdot \psi''_{t-s}(h(x,y,z,w))\,\di z\, \di y\\
&\quad + \frac{\beta_+}{2} \int_0^\infty g(y,r(y)-r(x)+w)\cdot \Big[\phi_{t-s}''(r(x)-r(y)) - \phi_{t-s}''(r(x)+r(y))\Big]\,\di y.
\ean
On the other hand, 
\ban
\frac{\partial}{\partial s} v(s,&x,w)  
= \frac{2}{\theta}\int_{-\infty}^{\infty}\int_{-\infty}^{r(y)-r(x)+w} g(y,z) \frac{\partial}{\partial s} \psi_{t-s}(h(x,y,z,w))\, \di z\, \di y\\
& + \beta_+\int_0^\infty g(y,r(y)-r(x)+w) \frac{\partial}{\partial s}\cdot\Big[\phi_{t-s}(r(x)-r(y)) - \phi_{t-s}(r(x)+r(y))\Big]\,\di y. 
\ean
Taking into account that $\frac{\partial}{\partial t} \phi_t(x) =\frac12 \phi''_t(x)$ and 
$\frac{\partial}{\partial t} \psi_t(x) = \frac12\psi''_t(x)$, we arrive at the desired equation
\ban
\Big(\frac{\partial}{\partial s} + \frac{1}{2\beta^2(x)}\cdot \frac{\partial^2}{\partial x^2} 
+ \frac{1}{\beta(x)}\cdot \frac{\partial^2}{\partial x\,\partial w} + \frac{1}{2}\cdot \frac{\partial^2}{\partial w^2} \Big) v(s,x,w) = 0,\  x>0.
\ean
Dealing similarly with the case $x<0$, we get 
\begin{equation*}\label{eq:pde-v}
\Big(\frac{\partial}{\partial s} +  \mathcal L\Big) v(s,x,w) = 0,\quad  x\neq 0,\  w\in \bR,
\end{equation*}
where
\begin{equation}\label{eq:Lf}
\mathcal L f(x,w) = \Big(\frac{\sigma^2(x)}{2}\cdot \frac{\partial^2}{\partial x^2} 
+ \sigma(x)\cdot \frac{\partial^2}{\partial x\,\partial w} + \frac{1}{2}\cdot \frac{\partial^2}{\partial w^2}\Big) f(x,w).
\end{equation}
Further, let $f\in C^\infty(\bR^2)$ have bounded support inside 
$\{(x,w)\colon  w<r(x)\}$, and 
\ban
p(s,x,w) = \frac{2\beta^2(x)}{\theta} \psi_s\Big(\frac{2x\beta^2(x) - w}{\theta}\Big)
\ean
as given by \eqref{eq:jointYBdensity}.
Denoting $\beta_- = \frac{1-\theta}{2}$ and using integration by parts, we write
\ban
\int_{-\infty}^\infty \beta^2(x)\cdot &p(s,x,w)\cdot f(x,w) \cdot \frac{\partial^2}{\partial x^2}v(s,x,w)\, \di x\\
&  = f(0,w)\cdot \Big[ \beta_-^2\cdot p(s,0-,x)\cdot \frac{\partial}{\partial x} v(s,0-,w) - \beta_+^2 \cdot p(s,0+,w)\cdot \frac{\partial}{\partial x}v(s,0+,w)\Big] \\
&-  \int_{-\infty}^{\infty} \beta^2(x)\cdot\frac{\partial}{\partial x}v(s,x,w)\cdot \frac{\partial}{\partial x}\Big(p(s,x,w)f(x,w)\Big)\, \di x.
\ean
We have 
\ban
& \beta_-^2\cdot p(s,0-,x)\frac{\partial}{\partial x}v(s,0-,w) - \beta_+^2\cdot p(s,0+,w)\frac{\partial}{\partial x}v(s,0+,w)\\ 
& \qquad = \frac{2}{\theta}\psi_t\Big(-\frac{ w}{\theta}\Big) \cdot 
\Big[ 2 \int_{-\infty}^{\infty}\beta^2(y)\cdot g(y,r(y) + w)\cdot \psi_{t-s}(|r(y)|)\,\di y \\
& \qquad\qquad+\beta^2_- \int_{-\infty}^0 g(y,r(y)+w)\cdot \Big[\phi_{t-s}'(-r(y)) - \phi_{t-s}'(r(y))\Big]\,\di y \\
& \qquad\qquad-\beta_+^2 \int_0^\infty g(y,r(y)+w)\cdot \Big[\phi_{t-s}'(-r(y)) - \phi_{t-s}'(r(y))\Big]\,\di y\Big] = 0,
\ean
since $\varphi'_t = \psi_t$ and $\psi_t(-x) = -\psi_t(x)$. 
Further, using integration by parts again and noting that $v$ is continuous, we get
\ban
&\int_{-\infty}^{\infty} \beta^2(x)\cdot \frac{\partial}{\partial x}v(s,x,w)\cdot \frac{\partial}{\partial x}\Big(p(s,x,w)f(x,w)\Big)\, \di x \\
&\qquad = v(s,0,w)\cdot \frac{\partial}{\partial x} f(0,w)\cdot \Big(\beta_-^2\cdot p(s,0-,w) - \beta_+^2\cdot p(s,0+,w)\Big)
\\ 
&\qquad\qquad  + v(s,0,w)\cdot  f(0,w)\Big(\beta_-^2\cdot \frac{\partial}{\partial x}p(s,0-,w) - \beta_+^2\cdot\frac{\partial}{\partial x}p(s,0+,w)\Big)
\\
&\qquad\qquad -  \int_{-\infty}^{\infty} \beta^(x)^2\cdot  v(s,x,w)\frac{\partial^2}{\partial x^2}\cdot \Big(p(s,x,w)f(x,w)\Big)\,\di x\\
&\qquad = - 4v(s,0,w)\cdot  f(0,w)\cdot \psi'_s\Big(\frac{w}{\theta}\Big)
\\&\qquad\qquad-  \int_{-\infty}^{\infty} \beta^2(x) \cdot  v(s,x,w)\cdot \frac{\partial^2}{\partial x^2}\Big(p(s,x,w)f(x,w)\Big)\, \di x.
\ean
Integrating with respect to $w$ then leads to 
\begin{align*}
\scprod{\beta^2 \frac{\partial^2}{\partial x^2}v(s), fp(s)} = \scprod{v(s), \beta^2 \frac{\partial^2}{\partial x^2}\big(fp(s)\big)} 
+ 4\int_{-\infty}^\infty v(s,0,w)f(0,w)\psi'_s\Big(\frac{w}{\theta}\Big)\, dw.
\end{align*}
Similarly, integrating by parts with respect to  $w$,
\ban
\scprod{\beta \frac{\partial^2}{\partial x\,\partial w}v(s), fp(s)} = 
 - \scprod{\beta \frac{\partial}{\partial x}v(s),\frac{\partial}{\partial w} \big(fp(s)\big)}.
\ean
and integrating by parts with respect to $x$, 
\ban
&\int_{-\infty}^\infty \beta(x)\cdot \frac{\partial}{\partial w}\big(p(s,x,w)f(x,w)\big) \frac{\partial}{\partial x}v(s,x,w)\, \di x\\
& \qquad = v(s,0,w)f(0,w)\Big(\beta_-\frac{\partial}{\partial w} p(s,0-,x) - \beta_+\frac{\partial}{\partial w}p(s,0+,w)\Big) \\
&\qquad\qquad - \int_{-\infty}^\infty \beta(x)\cdot v(s,x,w)\frac{\partial^2}{\partial x\,\partial w}\Big(p(s,x,w)f(x,w)\Big)
 -2 v(s,0,w)f(0,w) \psi'_s\Big(\frac{w}{\theta}\Big), 
\ean
so 
\ban
\scprod{\beta \frac{\partial^2}{\partial x\,\partial w}v(s), fp(s)} 
= \scprod{v(s), \beta \frac{\partial^2}{\partial x\,\partial w}\big(fp(s)\big)} - 2\int_{-\infty}^\infty v(s,0,w)f(0,w)\psi'_s\Big(\frac{w}{\theta}\Big)\,dw.
\ean

Finally, integrating by parts with respect to $w$ twice, 
\ban
\scprod{ \frac{\partial^2}{\partial w^2} v(s), fp(s)} = \scprod{ v(s), \frac{\partial^2}{\partial w^2} \big(fp(s)\big)}.
\ean
Summing everything, we get, with $\mathcal L$ given by \eqref{eq:Lf}, that
\be\label{eq:Libp}
\scprod{\mathcal L v(s), fp(s)\vphantom{\big)}} = \scprod{v(s),\mathcal L \big(fp(s)\big)}.
\ee

\section{Proof of Theorem~\ref{t:cov} for $\theta=\pm 1$}\label{a:petit}

Without loss of generality, let $\theta=1$, so that the skew Brownian motion $B^1$ is a non-negative reflected Brownian motion. 
The process $B^1$ has a density w.r.t.\ the Lebesgue measure 
\ba
p(t,z)&=\frac{\sqrt{2}}{\sqrt{\pi t}}\ex^{-\frac{z^2}{2t}},\quad z\geq 0.
\ea
On the time interval $t\in[0,T]$, consider the time reversed process $\bar B^1_t:=B^1_{T-t}$. Denote $b(s,z) = - \frac{\partial}{\partial z}\log p(s,z) = \frac{z}{t}$, $z\ge 0$, $t>0$.

\begin{lem}[Lemma 2.1, \cite{Petit-97}]
Define the process 
\ba
\bar W_t&:=
\bar B^1_t -\bar B^1_0-   L_t(\bar B^1) 
+\int_0^t  b (T-s,\bar B^1_s) \, \di s,\\
&=-B_T+B_{T-t}- L_t(B^1)+  L_{T-t}(B^1)
+\int_{T-t}^T  b (s, B^\theta_s) \, \di s,\quad t\in[0,T].
\ea
Then the process $\bar W$ is a $\rF^{\bar B^1}$- Brownian motion and
the time reversed process $\bar B^1$ is a solution to the SDE 
\ba
\label{e:barB1}
&\bar B^1_t= \bar B^1_0+ \bar W_t + L_t(\bar B^1) 
+\int_0^t b(T-s,\bar B^1_s) \, \di s.
\ea

Moreover, the time-reversed Brownian motion $\bar B_t=B_{T-t}$ satisfies
\ba
\label{e:barB}
\bar B_t&=B^1_{T-t}- L_{T-t}(B^1)=\bar B^1_{t} -  L_{t}(\bar B^1)
=   \bar B^1_0 + \bar W_t  
+\int_0^t b(T-s,\bar B^1_s) \, \di s .
\ea
\end{lem}

\begin{prp}
\label{t:covB1}
Let $f\in L^2_{\mathrm{loc}}(\bR )$ and let the reflected Brownian motion $B^1$
be the unique strong solution of the SDE \eqref{e:sdesbm}.
Then the quadratic variation
\ben
 [f(B^1),B]_t = \lim_{n\to\infty} \sum_{t_k\in D_n,t_k< t} 
 \big(f(B^1_{t_k})-f(B^1_{t_{k-1}})\big)(B_{t_k}-B_{t_{k-1}})
\een
exists as a limit in u.c.p.
Moreover, let $\{h_m\}_{m\geq 1}$ be a sequence of continuous functions such that for each $A>0$
\ba
\lim_{m\to \infty}\int_{-A}^A |h_m(x)-f(x)|^2\,\di x =0.
\ea
Then
\ba
\label{e:convbracket}
[h_m(B^1),B]_t\to  [f(B^1),B]_t, \quad m\to\infty,
\ea
in u.c.p.
\end{prp}
\begin{proof}
%

1. First we note that for a square integrable continuous function $h$, the quadratic variation 
\ban
{} [h(B^1),B]_t = \lim_{n\to\infty} \sum_{t_k\in D_n,t_k<t} \Big(h(B^1_{t_k})-h(B^1_{t_{k-1}})\Big)(B_{t_k}-B_{t_{k-1}})
\ean
exists as a limit in u.c.p. Indeed arguing as in \cite[Proposition 3.1]{FPSh-95}
we see that since $B$ is a semimartingale, and $h(B^1)$ is an adapted continuous process, the Riemannian sums converge to the It\^o integral in u.c.p.\
by \cite[Theorem 21, p.\ 64]{Protter-04}:
$$
\lim_{n\to\infty} \sum_{t_k\in D_n,t_k<t} \sum_{k=1}^n h(B^1_{t_{k-1}})(B_{t_k}-B_{t_{k-1}}) = \int_0^t h(B^1_s)\,\di B_s. 
$$

Let $T>0$ be fixed.
The time-reversed processes $(\bar B^1,\bar B)$ satisfy \eqref{e:barB1} and \eqref{e:barB} and are also
semimartingales so that for the backward Riemannian sums we get the u.c.p.\ convergence
\ban
\lim_{n\to\infty} \sum_{t_k\in D_n,t_k<t} &\sum_{k=1}^n h(B^1_{t_{k}})(B_{t_k}-B_{t_{k-1}}) 
= \lim_{n\to\infty} \sum_{t_k\in D_n,t_k<t} \sum_{k=1}^n h(\bar B^1_{T-t_{k}})(\bar B_{T-t_k}-\bar B_{T-t_{k-1}})\\ 
&=-\lim_{n\to\infty} \sum_{s_k\in D_n,s_{k+1}\geq t} h(\bar B^1_{s_{k}})(\bar B_{s_{k+1}}-\bar B_{s_{k}})
=-\int_{T-t}^{T} h(\bar B^1_s)\,\di \bar B_s\\
&= \int_{0}^{t} h( B^1_s)\,\di^* B_s.
\ean
Hence, the quadratic variation process exists as a difference
\ba
\label{e:qv}
{}[h(B^1),B]_t = \int_0^t h(B^1_s)\,\di^* B_s - \int_0^t h(B_s^1)\,\di B_s.
\ea

\noindent 
2. Let us extend the formula \eqref{e:qv} to locally square integrable functions. First note that by usual localization argument 
we can assume that $f\in L^2(\bR )$. Note that
$\E\int_0^T |f(B^1_t)|^2\,\di t<\infty$, so that the It\^o integral $I(t)=\int_0^t f(B^1_s)\,\di B_s$ exists.

Let $h$ be a square integrable continuous function.
Then the estimate $p(s,z)\leq \frac{1}{\sqrt s}$ for the density of $B^1_s$, the Doob inequality and the It\^o isometry yield 
\begin{align*}
 \E  \sup_{t\in[0,T]}\Big|\int_0^t h(B^1_s)\,\di B_s - \int_0^t f(B^1_s)\,\di B_s \Big|^2  & 
 \leq 4 \int_0^T \E  \big(f(B^1_t) - h(B^1_t)\big)^2 \,\di t \\
&  = 4\int_0^T \int_{0}^\infty \big(f(z) - h(z)\big)^2 p(t,z)\,\di z\,\di t\\
& \le  4  \|f - h\|^2_{L^2(\bR )}\cdot \int_0^T \frac{\di t}{\sqrt{t}}\le C \cdot \|f - h\|^2_{L^2(\bR )},
\end{align*}
for some constant $C>0$.
Let now
\ban
S_{n}(t) & := \sum_{t_k\in D_n, t_k\leq t} \big(h(B^1_{t_{k-1}}) - f(B^1_{t_{k-1}})\big) (B_{t_k}-B_{t_{k-1}}).
\ean
Similarly to the argument above we get
\begin{align*}
\E \sup_{t\in[0,T]} |S_{n}(t)|^2 
 & =   \E \sum_{t_k\in D_n} \big(h(B^1_{t_{k-1}}) - f(B^1_{t_{k-1}})\big)^2(t_k-t_{k-1})   \\
& \le C \cdot \sum_{t_k\in D_n,k>1} \frac{t_{k}-t_{k-1}}{\sqrt{t_{k-1}}}\|h-f\|^2_{L^2(\bR )}\\
&\leq C \cdot \|h-f\|^2_{L^2(\bR )}\cdot \sup_{n}\sum_{t_k\in D_n,k>1} \frac{t_{k}-t_{k-1}}{\sqrt{t_{k-1}}}
\leq C \cdot \|h-f\|^2_{L^2(\bR )}.
\end{align*}
As a result, we get
\ban
\limsup_{n\to\infty}\E \sup_{t\in[0,T]} \Big| \sum_{t_k\in D_n, t_k<t} h(B^1_{t_{k-1}})(B_{t_k}-B_{t_{k-1}}) -\int_0^t f(B^1_s)\,\di B_s \Big|^2 
\leq C \cdot \|h-f\|^2_{L^2(\bR )}.
\ean

3. To treat the backward integral recall again that the time-reversed process $(\bar{B}^1_t,\bar B_t) = (B^1_{T-t},B_{T-t})$ satisfies 
\eqref{e:barB1} and \eqref{e:barB} with the Brownian motion $\bar W$. Then for $h\in C(\bR )$ we can write 
\begin{align*}
&\int_0^t h(B_s^1)\, \di^* B_s  =  \int_{T-t}^T h(\bar B^1_s)\,\di \bar{B}_s 
= \int_{T-t}^T h(\bar{B}^1_s)\, \di \bar W_s 
+ \int_{T-t}^T h(\bar{B}^1_s) b(T-s,\bar{B}^1_s )\, \di s .
\end{align*}
Repeating the arguments from Step 2, we obtain the existence of the It\^o integral $\int_{T-t}^T f(\bar{B}^1_s)\,\di \bar W_s$ and the estimate
\ban
\limsup_{n\to\infty}\E \sup_{t\in[0,T]} \Big| \sum_{s_k\in D_n, s_{k+1}>t}  h(\bar B^1_{s_{k}})(\bar W_{s_{k+1}}-\bar W_{s_{k}})  - 
\int_{T-t}^T f(\bar{B}^1_s)\,\di \bar W_s  \Big|^2 \leq C \cdot \|h-f\|^2_{L^2(\bR )}.
\ean

It remains to establish an estimate for the Lebesque integral $\int_{0}^t f(B^1_s) b(s,B^1_s )\, \di s$:
\ban
\int_{0}^T \E\Big|f(B^1_s) b(s,B^1_s )\Big|\, \di s
&\leq \int_{0}^T \Big( \E |f(B^1_s)|^2 \cdot \E\Big|b(s,B^1_s )\Big|^2\Big)^{1/2}\, \di s.
\ean
Then
\ban
\E |f(B^1_s)|^2= \int_0^\infty |f(z)|^2 \cdot p(s,z)\,\di z\leq \frac{C}{\sqrt s}\|f\|_{L^2(\bR )}^2
\ean
and, recalling that $b(t,z) = \frac{z}{t}$,
\ban
\E\Big|b(s,B^1_s )\Big|^2
&=
\int_0^\infty \frac{z^2}{s^2}\cdot p(s,z)\,\di z
\leq \frac{C}{s^{5/2}} \int_0^\infty z^2 \ex^{-z^2/2s}\,\di z =\frac{C}{s}.
\ean
Hence
\ban
\int_{0}^T \E\Big|f(B^1_s)b(s,B^1_s )\Big|\, \di s\leq C \cdot \|f\|_{L^2(\bR )}\cdot \int_0^T\frac{\di s}{s^{3/4}}
\leq  C \cdot \|f\|_{L^2(\bR )}.
\ean
Furthermore for a continuous and square integrable $h$ 
\begin{equation}\label{eq:sumgb->intgbB1}
\begin{aligned}
&\sum_{t_k\in D_n, t_k<t} h(B_{t_{k}}^1)\int_{t_{k-1}}^{t_k}  b p(T-s,\bar{B}^1_s )     \, \di s  
\\&= - \sum_{t_k\in D_n, t_k<t} h(\bar B^1_{T-t_{k}})\int_{T-t_{k-1}}^{T-t_k} \frac{\bar B_s^1}{s}\, \di s 
 \to - \int_{T-t}^T  h(\bar B^1_{s})    \frac{\bar B_s^1}{s}     \, \di s 
\end{aligned}
\end{equation}
 uniformly on $[0,T]$ in probability as $n\to\infty$. 
Since
\ban
\delta_n = \max_{t_k\in D_n}\sup_{s\in[t_{k-1},t_k]} |h(B^1_{t_{k}}) - h(B_s^1)  |\to 0,\ n\to\infty,
\ean
almost surely and we can estimate
\begin{align*}
& \Big|\sum_{t_k\in D_n, t_k<t} h(B^1_{t_{k}})\int_{t_{k-1}}^{t_k} b(T-s,B_s^1)\, \di s 
- \int_0^t h(B_s^1) b(T-s,B_s^1)\,\di s\Big|
 \le \delta_n \int_0^T \frac{\bar B_s^1}{s}    \,\di s.
\end{align*}
Similarly to the calculations above, 
\ban
\E  \int_0^T\frac{\bar B_s^1}{s}      \, \di s  \leq  C\int_0^T \frac{\di s}{\sqrt s} \le C.
\ean
Therefore
\ban
\sup_{t\in[0,T]}\Big|\sum_{t_k\in D_n, t_k<t} h(B^1_{t_{k}})\int_{t_{k-1}}^{t_k} b(T-s,B_s^1)\, \di s 
- \int_0^t h(B_s^1) b(T-s,B_s^1)\,\di s\Big|
\to 0
\ean
in probability as $n\to\infty$. 

Eventually taking a sequence $h_m\in C(\bR )$ converging to $f$ in $L^2(\bR )$  we arrive at 
\eqref{eq:sumgb->intgbB1}.
Combined with our previous findings, this leads to
\ban
&[f(B^1),B]_t=\lim_{n\to \infty} \sum_{t_k\in D_n,t_k<t} \Big(f(B^1_{t_k})-f(B^1_{t_{k-1}})\Big)(B_{t_k}-B_{t_{k-1}})
\\&= \int_0^t f(B^1_{s})\,\di B_s - \int_0^t f(B^1_{s})\,\di^* B_s
\ean
uniformly on $[0,T]$ in probability. Since $T>0$ is arbitrary, this  means precisely that the desired u.c.p.\ convergence holds.

\end{proof}

\end{document}